\documentclass[11pt,reqno]{amsart}
\usepackage{amsmath}
\usepackage{amssymb}
\usepackage{amscd}
\usepackage{epsfig}
\usepackage{amsxtra}
\usepackage{pifont}
\usepackage{mathabx}
\usepackage{MnSymbol}
\usepackage{accents}
\usepackage{scalerel,stackengine}
\usepackage{xcolor}
\usepackage{wasysym}

\stackMath
\newcommand\reallywidehat[1]{%
\savestack{\tmpbox}{\stretchto{%
  \scaleto{%
    \scalerel*[\widthof{\ensuremath{#1}}]{\kern-.6pt\bigwedge\kern-.6pt}%
    {\rule[-\textheight/2]{1ex}{\textheight}}
  }{\textheight}%
}{0.5ex}}%
\stackon[1pt]{#1}{\tmpbox}%
}

\newcommand\reallywidecheck[1]{%
\savestack{\tmpbox}{\stretchto{%
  \scaleto{%
    \scalerel*[\widthof{\ensuremath{#1}}]{\kern-.6pt\bigwedge\kern-.6pt}%
    {\rule[-\textheight/2]{1ex}{\textheight}}
  }{\textheight}%
}{0.5ex}}%
\stackon[1pt]{#1}{\scalebox{-1}{\tmpbox}}%
}

\numberwithin{equation}{section}

\allowdisplaybreaks

\newcommand{\supp}{\mbox{\rm supp}}

\newcommand{\RR}{{\mathbb R}}

\newcommand{\ZZ}{{\mathbb Z}}
\newcommand{\CC}{{\mathbb C}}
\newcommand{\TT}{\mathbb T}
\newcommand{\NN}{\mathbb N}

\newcommand{\M}{{\mathcal M}}

\newcommand{\cM}{{\mathcal M}}

\newcommand{\cA}{{\mathcal A}}
\newcommand{\cB}{{\mathcal B}}
\newcommand{\cC}{{\mathcal C}}
\newcommand{\cH}{{\mathcal H}}

\newcommand{\dd}{\mbox{d}}

\newcommand{\Cu}{C_{\mathsf{u}}}
\newcommand{\Cc}{C_{\mathsf{c}}}

\newcommand{\eps}{\varepsilon}

\newcommand{\WAP}{\mathcal{WAP}}

 \newtheorem{theorem}{Theorem}[section]
 \newtheorem{lemma}[theorem]{Lemma}
 \newtheorem{proposition}[theorem]{Proposition}
 \newtheorem{corollary}[theorem]{Corollary}

 \newtheorem{definition}[theorem]{Definition}
 \newtheorem{example}[theorem]{Example}
  \newtheorem{remark}[theorem]{Remark}
    \newtheorem{question}[theorem]{Question}

\newcommand{\lb}{\text{\textlquill} }
\newcommand{\rb}{\text{\textrquill} }

\begin{document}
\title[Diffraction and orthogonality]{Diffraction as a unitary representation and the orthogonality of measures with respect to the reflected Eberlein convolution}

\dedicatory{ We dedicate this work to Michael Baake on the occasion of his 65\textsuperscript{th} birthday.}

\author{Daniel Lenz}
\address{Mathematisches Institut, Friedrich Schiller Universit\"at Jena, 07743 Jena, Germany}
\email{daniel.lenz@uni-jena.de}
\urladdr{http://www.analysis-lenz.uni-jena.de}

\author{Nicolae Strungaru}
\address{Department of Mathematical Sciences, MacEwan University \\
10700 -- 104 Avenue, Edmonton, AB, T5J 4S2, Canada\\
and \\
Institute of Mathematics ``Simon Stoilow''\\
Bucharest, Romania}
\email{strungarun@macewan.ca}
\urladdr{https://sites.google.com/macewan.ca/nicolae-strungaru/home}

\begin{abstract}
We discuss how the diffraction theory of a single
translation bounded measure or a family of such measures can be
understood within the framework of unitary group representations. This allows us to prove an orthogonality feature of measures whose
diffractions are mutually singular. We apply this to study dynamical systems, the refined Eberlein decomposition and validity of a Bombieri--Taylor type result in a rather general context. Along the way we also use our approach to  (re)prove various characterisations of pure point diffraction.
\end{abstract}

\maketitle

\section*{Introduction}
This article is concerned with mathematical diffraction theory.  A core object in mathematical diffraction theory is the autocorrelation of a measure. This autocorrelation is an averaged quantity.  The theory can be developed on arbitrary locally compact Abelian groups and this is how we will proceed below. In the case of the integers the autocorrelation deals with means of the form
\[
\lim_N \frac{1}{N}\sum_{k=1}^N f(k) \overline{f(k-j)} = :\gamma_f (j)
\]
for a bounded function $f$ on $\ZZ$ and $j\in\ZZ$.

As was observed recently, mathematical diffraction theory can conveniently be phrased with the help of the reflected Eberlein convolution \cite{LSS3}. Specifically, the autocorrelation of a measure is the  reflected Eberlein convolution of the measure and itself.  As was also noted in \cite{LSS3}, this  reflected Eberlein convolution provides a certain inner product like structure to the space of measures.

The starting point of this article is the  realization  that the reflected Eberlein convolution is not only somewhat similar to an inner product,  but that one can rather construct a proper Hilbert space together with a unitary representation of the underlying group out of the reflected  Eberlein convolution. This allows us to study orthogonality with respect to the reflected Eberlein convolution. Our main result  gives orthogonality of measures when their diffractions are mutually singular.  Having established the  unitary  representation we obtain this  result rather directly from Stone theorem. For bounded functions $f,g$ on the integers the result gives that
\[
\lim_N \frac{1}{N}\sum_{k=1}^N f(k) \overline{g(k)} = 0
\]
must necessarily hold whenever the diffraction measures $\widehat{\gamma_f}$ of  $f$ and $\widehat{\gamma_g}$ of $g$ are mutually singular.

The main result a allows for a  variety of  applications. One
application concerns what is sometimes called the Bombieri--Taylor
conjecture. Another application concerns orthogonality of dynamical
dynamical systems $(X,G,m)$ and $(X',G,m')$. If for such systems   the spectral measures of  $f \in C(X)$ and $f'\in C(X')$ are mutually singular, then the functions $t \to f(t x)$ and $t\mapsto f'(t x')$
are orthogonal with respect to the reflected Eberlein convolution
for almost surely all $x\in X$ and $x'\in X'$.  A third application concerns what is known as refined Eberlein decomposition.  Besides these applications of the orthogonality result we can also use the underlying unitary representation to (re)prove various characterizations of pure point diffraction. This gives  in particular a new and unifying perspective on  results achieved during the last twenty years.

The article is organized as follows: In Section
\ref{section-setting} we present the basic setting of locally
compact Abelian groups and the associated notions needed in the
remaining part of the article. Section \ref{section-eberlein} then
features the reflected Eberlein convolution and its basic
properties. The reflected Eberlein convolution is defined via a limit and we discuss existence of this limit in  Section \ref{section-existence}.
The construction of the unitary representation, first
out of functions, and, then out of measures admitting a reflected
Eberlein convolution is done in Section \ref{section-construction}.
The mentioned characterisation of pure point diffractions is then derived in Section
\ref{section-main}, while our main result on orthogonality is proven in Section~\ref{section-main}.
The subsequent three sections then
discuss the mentioned applications to Bombieri--Taylor conjecture and
to dynamical systems as well as an application to the refined
Eberlein decomposition. The final section is devoted to the
reflected Eberlein convolution with a Besicovitch almost periodic
measure. We characterize in particular those measures which have
vanishing reflected Eberlein convolution with all Besicovitch almost
periodic measures.

\section{The setting}\label{section-setting}
Throughout this paper $G$ is a  locally compact Abelian group
(LCAG).  The group operations are written additively and
the neutral element is denoted by $0$.  Integration of $f$ with
respect to Haar measure is denoted by $\int_G f \dd s$. The Haar
measure of a measurable subset $A$ of $G$ is also denoted by $|A|$.

We denote by $\Cu(G)$ the space of uniformly
continuous bounded functions on $G$ and by $\Cc(G)$ the subspace of
$\Cu(G)$ consisting of compactly supported continuous functions.
The space $\Cu(G) $ is equipped with the supremum norm $\|\cdot\|_\infty$ defined by $\|f\|_\infty :=\sup\{|f(t)| : t\in G\}$. The support $\supp (\varphi)$ of $\varphi \in\Cc (G)$ is the smallest compact set outside of which $\varphi$ vanishes.
The spaces $\Cu(G)$ and $\Cc (G)$
admit  the following operators
\[
f^\dagger(x)=f(-x) \mbox{ and } \tilde{f}(x)= \overline{f(-x)} \,.
\]
Moreover, each $t \in G$ induces a translation operator on these spaces via
\[
\tau_t f(x)=f(x-t) \,.
\]

A \emph{Radon measure} on  $G$  is a linear functional $\mu : \Cc(G)
\to \CC$ with the property that for each compact set $K \subseteq G$
there exists  $C_K\geq 0$ such that all functions $\varphi \in
\Cc(G)$ whose support is contained in $K$ satisfy
\[
\left| \mu(\varphi) \right| \leq C_K \| \varphi \|_\infty.
\]
We will often write $\int_{G} \varphi(t) \dd \mu(t):= \mu(\varphi)$.

The operators ${ \, }^\dagger, \tilde{ \, }, \tau_t$ extend
naturally to measures via
\[
\mu^\dagger(\varphi)= \mu(\varphi^\dagger) \,;\, \tilde{\mu}(\varphi) = \overline{\mu(\tilde{\varphi})} \,;\,(\tau_t \mu)(\varphi)= \mu(\tau_{-t} \varphi) \,.
\]

To any Radon measure $\mu$ there exists a unique positive measure
$\nu$ and a measurable $h : G\longrightarrow \CC$ with $|h| =1$
such that
\[
\mu (\varphi) = \int \varphi h \dd \nu
\]
holds  for all $\varphi \in\Cc (G)$ \cite[Thm.~6.5.6]{Ped}. The measure $\nu$ is called
the \textit{total variation} of $\mu$ and henceforth denoted by $|\mu|$. The measure $\mu$ is called
finite if $|\mu|(G)<\infty$ holds.

Whenever $A$ is a Borel subset of $G$ we define the restriction $\mu|_A$ of $\mu$ to
$A$ to be the measure satisfying
\[
\mu|_A (\varphi) :=\int_A \varphi h \dd|\mu|
\]
for all $\varphi \in\Cc (G)$.

\smallskip

The dual group $\widehat{G}$ of $G$ is the set of all continuous group homomorphisms $\xi :  G\longrightarrow \TT$. Here, $\TT$ is the group of  complex numbers with modulus $1$ (equipped with multiplication).  The dual group $\widehat{G}$  is a locally compact Abelian group in a natural way. The \textit{Fourier transform} $\widehat{f}$ of $f \in L^1(G)$ is the function on $\widehat{G}$ with
\[
\widehat{f}(\xi) = \int_{G} \overline{\xi(t)} f (t) \dd t.
\]
Moreover, whenever $\sigma$ is a finite positive measure on
$\widehat{G}$, we can define the inverse Fourier transform
$\widecheck{\sigma}$ of $\sigma$ as the function on $G$ given by
\[\widecheck{\sigma}(t) = \int_{\hat{G}} \xi(t)  \dd\sigma(\xi) \,.\]

\medskip

The convolution $\varphi \ast\psi$ of $\varphi, \psi\in\Cc (G)$ is
the function on $G$ defined by
\[
\varphi \ast \psi (t) := \int_{ G} \varphi (s) \psi (t-s) \dd s \,.
\]
The convolution  $\mu\ast \varphi$ between  a
measure $\mu$ and a function $\varphi \in \Cc(G)$ is the function on
$G$ defined by
\[
\mu\ast \varphi (t) = \int_{G} \varphi(t-s) \dd \mu(s) = \mu( (\tau_t \varphi)^\dagger)\,.
\]
It is easy to see that
\[
\tau_t ( \mu\ast \varphi) =(\mu \ast \tau_t \varphi)= (\tau_t \mu)
\ast \varphi  \,
\]
for all $t\in G$, $\varphi \in \Cc(G)$ and any measure $\mu$.

The convolution $\mu\ast \nu$ between two finite measures is the
measure given by
\[
\mu \ast\nu (\varphi) = \int_{G} \int_{ G} \varphi (s+t) \dd\mu(s) \dd\nu
(t) \qquad \forall \varphi \in \Cc(G) \,.
\]

A measure $\gamma$ is called \textit{positive definite} if
\[
\gamma\ast \varphi \ast \widetilde{\varphi} (0)\geq 0
\]
holds for all $\varphi \in\Cc (G)$. This is equivalent to $\gamma\ast \varphi \ast \widetilde{\varphi}$ being a continuous positive definite function for all $\varphi \in \Cc(G)$ \cite{BF}.

Every positive definite measure $\gamma$ admits a (unique) positive  measure $\widehat{\gamma}$ on $\widehat{G}$ such that
\[
\gamma\ast \varphi \ast \widetilde{\varphi}(0) = \int_{\widehat{G}} |\widehat{\varphi}(\xi) |^2 \dd\widehat{\gamma}(\xi)
\]
for all  $\varphi \in \Cc (G)$. The measure $\widehat{\gamma}$ is called the \textit{Fourier transform} of $\gamma$ (see
\cite{ARMA1,BF,MoSt} for details). From the definition and polarisation  we easily find
\[
\gamma\ast \varphi \ast \widetilde{\psi}(0) = \int_{\widehat{G}} \widehat{\varphi}(\xi) \overline{\widehat{\psi}(\xi)} \dd\widehat{\gamma}(\xi)
\]
for all $\varphi,\psi\in\Cc (G)$. Applying this with $\psi = \tau_{-t} \varphi$ we obtain the following statement.

\begin{proposition}\label{char-ft-gamma} Let $ \gamma$ be a positive definite measure on $G$. Then,
\[
\gamma  \ast \varphi \ast \widetilde{\varphi}(t) =\int_{\widehat{G}} \xi (t) |\widehat{\varphi}|^2 (\xi) \dd \widehat{\gamma}(\xi)
\]
holds for all $t\in G$ and $\varphi \in\Cc (G)$. \qed
\end{proposition}

A measure $\mu$ is called \emph{translation bounded} if for one
(any) non-empty open $V$ with compact closure
\[
\|\mu\|_{V}:=\sup_{t\in G} |\mu| (t + V) <\infty
\]
holds. This is equivalent to  $\mu\ast \varphi$ being  bounded for
all $\varphi \in\Cc (G)$ \cite{ARMA1}. Note that for a translation bounded $\mu$
the function $\mu \ast\varphi$  belongs to $\Cu (G)$ \cite{ARMA1} (compare
appendix for further discussion of translation bounded measures).

\smallskip

Let us complete the section by briefly discussing $\sigma$-compactness, metrisability and second countability of $G$.

\begin{remark}[Second countable LCAG]
\begin{itemize}
  \item[(a)] If $X$ is any locally compact Hausdorff space, and $B$ is any basis of open sets for the topology, then the set
  \[
  B_{c}:=\{ U \in B : \bar{U} \mbox{ is compact } \}
  \]
  is also a basis of open sets for the topology. This immediately implies that any second countable locally compact Hausdorff space is $\sigma$-compact.
  \item[(b)] Any locally compact group is a normal topological space. Therefore, by the Urysohn Metrisation Theorem, any second countable LCAG is metrisable.
  \item[(c)] If $G$ is a metrisable LCAG and $K \subseteq G$ is compact, it is easy to show that there exists a countable dense set $D \subseteq G$.
  In particular, any $\sigma$-compact and metrisable group has a countable dense subset. This combined with metrisability gives that any $\sigma$-compact and metrisable LCAG is
  second countable.
  \item[(d)] Points (a-c) imply that a LCAG $G$ is second countable if and only if $G$ is $\sigma$-compact and metrisable.
  \item[(e)] By \cite[Thm.~4.2.7]{ReiSte} and Pontryagin duality (see \cite[Thm.~4.2.11]{ReiSte}) we have the following equivalences:

  \begin{itemize}
    \item{} $G$ is $\sigma$-compact if and only if $\widehat{G}$ is metrisable.
    \item{} $G$ is metrisable if and only if $\widehat{G}$ is $\sigma$-compact.
  \end{itemize}

\item[(f)] By (d) and (e)   the group  $G$ is second countable if and only if $\widehat{G}$ is second countable.

\end{itemize}
\end{remark}

\section{The reflected Eberlein convolution}\label{section-eberlein}
In this section we discuss the reflected Eberlein convolution. This discussion is essentially taken from \cite{LSS3}. However, as \cite{LSS3} makes the additional assumptions that the group is $\sigma$-compact with countable basis of topology some (small) adjustments are necessary. In order to ease the approach for the reader we  present a rather complete treatment in this section.

Let $A$ be a subset of $G$. Then, for
 each compact set $K \subseteq G$  the \textit{$K$-boundary $\partial^{K} A$} of  $A$ is defined as
\[
\partial^{K} A := \bigl( \overline{A+K}\setminus A\bigr) \cup
\bigl((\left(G \backslash A\right) - K)\cap \overline{A}  \bigr)
\]

A net $(A_i)_{i \in I}$ of open  subsets of $G$ with compact closure is called
a \textit{van Hove net} if
\[
\lim_{i } \frac{|\partial^{K} A_{i}|}{|A_{i}|}  =  0
     \]
holds for any compact $K\subset G$. Note that every LCAG admits a
van Hove net (see for example \cite[Prop.~5.10]{PRS}).

With the help of a van Hove net we can define an averaged version of
the convolution as follows.

\begin{definition}[Reflected Eberlein convolution of measures]
Let $\cA$ be a van Hove net. Let translation bounded measures $\mu$
and $\nu$ be given.  If  the limit
\[
 \lim_{i} \frac{1}{|A_i|} (\mu|_{A_i})\ast\widetilde{(\nu|_{A_i})}
\]
exists in the vague topology it is called the \textit{reflected
Eberlein convolution} of $\mu$ and $\nu$ with respect to the van
Hove net  $\cA$ and denoted by $\lb \mu , \nu \rb_{\cA}$.
\end{definition}

Let us now list the basic properties of the reflected Eberlein
convolution of measures.

\begin{lemma}\label{thm:rebe props}
Let $\mu, \nu$  be translation bounded measures and let $\cA$ be a van Hove net. Then, the following assertions hold:
\begin{itemize}
  \item[(a)] There exists an index $i_0$ and a vaguely compact set $X \subseteq \cM^\infty(G)$ such
that, for all $i \geq i_0 $ we have  $\frac{1}{|A_i|}
(\mu|_{A_i})\ast\widetilde{(\nu|_{A_i})} \in X$. In particular, the net
$\frac{1}{|A_i|} (\mu|_{A_i})\ast\widetilde{(\nu|_{A_i})}$ always
admits a convergent subnet.
\item[(b)] Assume that $\lb \mu , \nu \rb_{\cA}$ exists. Then,
     $\lb \tau_t \mu , \nu \rb_{\cA}$ exists for all $t\in G$ and
\[
\lb \tau_t \mu , \nu \rb_{\cA}= \lb \mu , \tau_{-t}\nu \rb_{\cA} =\tau_t \lb \mu , \nu \rb_{\cA}
\]
    holds. In particular, $\lb \tau_t \mu,\tau_t\mu\rb_{\cA} = \lb\mu,\nu\rb_{\cA}$ holds.
  \item[(c)] Assume that $\lb \mu , \nu \rb_{\cA}$ exists. Then, $\lb \nu , \mu \rb_{\cA}
$ exists and  satisfies
\[
\lb \nu , \mu \rb_{\cA}=\widetilde{\lb \mu , \nu \rb_{\cA}} \,.
\]
  \item[(d)] Assume that $\lb \mu , \nu \rb_{\cA}$ and $\lb \sigma , \nu \rb_{\cA}$ exist. Then, for all $a,b \in \CC$  also  $\lb a\mu+b \sigma , \nu \rb_{\cA}$ exists and
  \[
  \lb a\mu+b \sigma , \nu \rb_{\cA}=a\lb \mu , \nu \rb_{\cA}+b\lb \sigma  , \nu \rb_{\cA}
  \]
holds.
  \item[(e)] If $\lb \mu , \mu \rb_{\cA}$ exists then it is positive definite.
\end{itemize}
\end{lemma}
\begin{proof} \textbf{(a)} This is discussed in the appendix.

\textbf{(b), (c), (d)} are straightforward.

\textbf{(e)} The measure $\frac{1}{|A_i|}
(\mu|_{A_i})\ast\widetilde{(\mu|_{A_i})}$ is positive definite as for any $\varphi\in\Cc(G)$ we find
\[
\frac{1}{|A_i|}
(\mu|_{A_i})\ast\widetilde{(\mu|_{A_i})}\ast \varphi \ast \widetilde{\varphi}(0) = \frac{1}{|A_i|}\int |\psi (s)|^2  \dd s \geq 0
\]
with $\psi:=\mu|_{A_i} \ast \varphi$.
 Taking the limit we then find the desired statement (e).
\end{proof}

\begin{definition}[Autocorrelation and diffraction of $\mu$]
Let $\mu$  be a translation bounded measure and let $\cA$ be a van
Hove net. If it exists,  the measure $\lb \mu , \mu \rb_{\cA} $ is called the autocorrelation of $\mu$ and denoted by $\gamma_\mu$.
Then,  the positive measure
$\reallywidehat{\gamma_{\mu}}$ is called the \emph{diffraction measure of
$\mu$}.
\end{definition}

Let us also recall the concept of Fourier--Bohr coefficients.

\begin{definition}[Fourier--Bohr coefficient] Let $\mu$  be a translation bounded measure, let $\cA$ be a van Hove net and let $\chi \in \widehat{G}$. If it exists,  the
\[
\lim_{i} \frac{1}{|A_i|} \int_{A_i}
\overline{\chi(t)} \dd \mu(t) \,.
\] is called the  \emph{Fourier--Bohr coefficient} of $\mu$ and denoted by
 $a_{\chi}^\cA(\chi)$.
\end{definition}

Fourier--Bohr coefficients can be expressed via the reflected Eberlein convolution as follows. The proof is straightforward.

\begin{proposition}\label{prop-FB} Let $\mu$  be a translation bounded measure, let $\cA$ be a van Hove net  and let $\chi \in \widehat{G}$. Then, the Fourier--Bohr coefficient $a_{\chi}^\cA(\chi)$ exists if and only if
$ \lb \mu , \chi \rb_{\cA}$ exists. Moreover, in this case $\lb \mu
, \chi \rb_{\cA}$ is the absolutely continuous measure with density
function $a_{\chi}^\cA(\mu) \chi$. \qed
\end{proposition}

We now turn to functions and their means. This will provide a
further understanding of the Eberlein convolution.

\begin{definition}[The mean $M_\cA$]
Let $\cA$ be a van Hove net and $f$ a bounded measurable function on
$G$. If the limit
\[\lim_i \frac{1}{|A_i|} \int_{A_i} f(t) \dd t
\,,
\]
exists, it is called the \emph{mean} of $f$ with respect to
$\cA$ and denote by $M_{\cA}(f)$.
\end{definition}

Clearly,  the mean is $G$-invariant, i.e. for all bounded measurable
function on $G$ such that $M_{\cA}(f)$ exists, and all $t \in G$,
the mean $M_{\cA}(\tau_t) f$ exists and satisfies
\[M_{\cA}(\tau_t f) =M_{\cA} (f) \,.\]

\smallskip

Taking the mean is similar to integration with respect to a
(probability) measure. Accordingly, the mean allows one to define
analogues of convolution.

\begin{definition}[Reflected Eberlein convolution of functions]
Let $\cA$ be a van Hove net. Let  $f,g\in \Cu(G)$  be given. If for
each $t \in G$ the mean $M_\cA(f \overline{\tau_t g })$ exists,  the function
\[
\lb f , g \rb_{\cA} : G\longrightarrow \CC, t \mapsto M_\cA (f
\overline{\tau_t g}),
\]
is called the
\textit{reflected Eberlein convolution} of $f$ and $g$ with respect
to $\cA$.
\end{definition}

We note that the reflected Eberlein convolution of $f$ and $g$ is
given by
\[
 \lb f , g \rb_{\cA}(t) = \lim_{i} \frac{1}{|A_i|} \int_{A_i} f(s) \overline{g(s-t)} \dd s \, .
\]
Hence, the reflected Eberlein convolution can be understood  as an
averaged version of the convolution of $f$ and $\tilde{g}$. The
relationship between the reflected Eberlein convolution of measures
and of functions is discussed  in the subsequent lemma.

\begin{lemma}\label{lemma-relationsship}  Let $\cA$ be a van Hove net. Let $\mu,\nu$ be
translation bounded measures. Then, the following statements are
equivalent:

\begin{itemize}
\item[(i)] The reflected Eberlein convolution $\lb \mu,\nu\rb_{\cA}$
exists.

\item[(ii)] For all $\varphi,\psi \in \Cc (G)$ the reflected
Eberlein convolution $\lb \mu \ast \varphi, \nu \ast \psi\rb_{\cA}$
exists.

\item[(iii)] For all $\varphi,\psi \in \Cc (G)$  the mean  $M_{\cA} (\mu \ast
\varphi \, \cdot \, \overline{\nu\ast\psi})$ exists.

\end{itemize}
If one of the equivalent conditions (i), (ii) and (iii) is valid the
equalities
\[
  \lb \mu \ast\varphi , \nu\ast\psi \rb_{\cA} (t) =\lb \mu , \nu
    \rb_{\cA}\ast\varphi\ast\tilde{\psi} (t) = M_{\cA} (\mu \ast \varphi
   \; \overline{\nu \ast \tau_t \psi} )
   \]
    hold for all $\varphi,\psi \in\Cc (G)$ and all $t\in G$.
\end{lemma}
\begin{proof} We first discuss the equivalence between (ii) and (iii): By definition we have
\[
\lb \mu \ast \varphi, \nu\ast \psi\rb_{\cA}(t) = M_{\cA} (\mu\ast
\varphi  \; \overline{\tau_t (\nu\ast \psi}) ) =
M_{\cA} (\mu \ast \varphi \overline{\nu\ast \tau_t \psi})
\]
(if the corresponding limits exist). This easily gives the desired equivalence. We now turn to proving the equivalence between (i) and (iii). Along the way we will establish the given formulae:   For $A\subset G$ and $f : G\longrightarrow\CC$ we write $f|_A$ for
the functions which agrees with $f$ on $A$ and is zero outside of
$A$.

We start by comparing $\mu|_A \ast \varphi$ and $\mu\ast \varphi|_A$
for $\varphi \in\Cc (G)$:

Let $K$ be a compact set containing the neutral element of $G$ and
the support of $\varphi$.

For $t\in G$ with $t\notin A+ K$ we find
\[
(\mu|_A \ast \varphi) (t) = 0 = (\mu\ast \varphi) |_A (t) \,.
\]

For $t\in G$ with $t - K\subset A$ we find
\[
(\mu|_A \ast \varphi) (t) = \int_A \varphi (t-s) \dd\mu (s) = \int \varphi (t-s) \dd\mu = (\mu\ast \varphi)|_A (t) \,.
\]
Any remaining $t\in G$ belongs to $\partial^K A$.  So, we can summarize that $\mu|_A
\ast \varphi $ and $\mu\ast \varphi|_A$ agree outside of $\partial^K
A$.  From this  together with
\[
\mu|_A \ast \widetilde{\nu|_A}|
\ast \varphi \ast \widetilde{\psi} (0) = \int_G (\mu|_A \ast \varphi)
(s) \overline{(\nu|_A \ast \psi)} (s) \dd s
\]
and
\[  \int_A \mu\ast \varphi (s) \overline{\nu\ast \psi(s)}
\dd s =  \int_G (\mu\ast \varphi)|_A (s)  \overline{\nu\ast (\psi|_A) (s) } ds
\] we obtain from the van Hove property
\[
\frac{1}{|A_i|}\left|\mu|_{A_i} \ast \widetilde{\nu|_{A_i}}|
\ast \varphi \ast \widetilde{\psi} (0) -\int_{A_i} \mu\ast \varphi
(s) \overline{\nu\ast \psi}(s) \dd s\right|\to 0.
\]
Hence, existence of the mean $M_{\cA} (\mu\ast \varphi
\overline{\nu\ast \psi})$ for all $\varphi,\psi \in\Cc (G)$ is
equivalent to existence of the limit

\[
\lim_i \frac{1}{|A_i|} \mu|_{A_i} \ast \widetilde{\nu|_{A_ni}}
\ast \varphi \ast \widetilde{\psi} (0) \,.
\]
The latter in turn is
equivalent to vague convergence of $\frac{1}{|A_i|} \mu|_{A_i} \ast
\widetilde{\nu|_{A_i}}|$ (as discussed in the appendix).  This finishes the proof.
\end{proof}

As it is both  instructive and already interesting we point out the
following situation which is  covered by our setting.

\begin{example}[The case of $\ZZ$]  Let $G=\ZZ$ be the
group of integers (with addition). Then, $\cA$ with $A_n
=\{1,\ldots, n\}$ is a canonical choice of van Hove sequence. We can
identify the measures $\mu$ on $\ZZ$ with the function  - again
denoted by $\mu$ - with $\mu (k) :=\mu(\{k\})$. Then,
$\lb\mu,\nu\rb_{\cA}$ exists if and only if for each $j\in\ZZ$ the
limit
\[
\lim_n \frac{1}{n} \sum_{k=1}^n  \mu(k) \overline{\nu(k-j)}
\]
exists. \end{example}

\section{Existence of the reflected Eberlein convolution}\label{section-existence}.
The reflected Eberlein convolution is defined as a limit. So, a natural question concerns existence of this limit.  Here we discuss how existence of the limit  can always be achieved if one replaces the original van Hove net by a subnet and we also discuss how we can  work with van Hove sequences instead of nets under suitable countability assumptions on the topology.

We start with the following result. The result is certainly well-known. We include a proof for the convenience of the reader in the appendix.

\begin{theorem}[Existence of universal van Hove net]\label{thm:univ}
Every LCAG admits a van Hove net $\cA$ with the following properties:
\begin{itemize}
  \item[(a)]
	For all $f \in L^\infty(G)$ and $\chi \in \widehat{G}$ the Fourier--Bohr coefficient $a_{\chi}^\cA(f)$ exists.
  \item[(b)] For all $f,g \in L^\infty(G)$,  $\lb f,g \rb_{\cA}$ exists.
  \item[(c)] For all $\mu,\nu \in \cM^\infty(G)$,  $\lb \mu,\nu \rb_{\cA}$ exists.
  \item[(d)] For all $\mu \in \cM^\infty(G)$,  the autocorrelation $\gamma = \lb \mu ,\mu \rb_{\cA}$ exists with respect to $\cA$.
\end{itemize}
Furthermore, any van Hove net $\cB$  has a subnet $\cA$ with these properties.
\qed
\end{theorem}

We next discuss how van Hove nets can be replaced by van Hove sequences under certain circumstances.

First note that a van Hove sequence exists if and only if the group $G$ is $\sigma$-compact \cite[Prop.~B.6]{SS}. In this case, we  can  therefore replace van Hove nets by
van Hove sequences in the definition of the Eberlein convolution of measures and functions.
Then, the results in the subsequent part of this article   all hold with van Hove
net replaced by van Hove sequence. Note, however, we will still need
to use van Hove nets in the proofs whenever we apply  (a) of Lemma  \ref{thm:rebe props}.
The reason is that in general the convergent subnet appearing in (a) of that  lemma  will not be a subsequence.

We now turn to the case that the group is second countable. Then it is $\sigma$-compact and  we can therefore work with van Hove sequences. Moreover, the  subnet appearing in (a) of Lemma  \ref{thm:rebe props} can be chosen as a subsequence. Hence, all subsequent considerations of this article (statements and proofs)  are valid
with van Hove net replaced by van Hove sequence for second countable locally compact Abelian groups.
Moreover, in this case we have a - so to speak - countable analogue of the above result on universal van Hove sequences. Specifically, whenever $\cB$ is a van Hove sequence and $M$ is any set of translation bounded measures, for any pair $\mu, \nu$ of measures in $M$, the reflected Eberlein convolution $\lb \mu, \nu
\rb_{\cC}$ exists along some subsequence $\cC$ of $\cB$. If $M$ is
countable, then a standard diagonalisation argument (or Tychonoff's
Theorem) shows that given any van Hove sequence $\cA$, there exists
some subsequence $\cC$ such that, that for all $\mu, \nu \in M$ the
reflected Eberlein convolution $\lb \mu, \nu \rb_{\cC}$ exists.

In the context of working with van Hove sequences instead of van Hove nets we also record the following three results.

\begin{lemma}\label{lem4} Let $f \in L^\infty(G), g \in \Cu(G)$ and let $(A_i)_{i \in I}$ be a van Hove net. Then, the set of $t \in G$ for which the net
\begin{equation*}
\frac{1}{|A_i|} \int_{A_i} f(s) \overline{g(s-t)} \dd s.
\end{equation*}
converges is closed in $G$.
\end{lemma}
\begin{proof} For any $i\in I$ we define
\[
F_i : G\longrightarrow \CC, F_i (t) = \frac{1}{|A_i|}\int_{A_i} f(s) \overline{g(s-t)} \dd s \,.
\]
Then, the family $(F_i)$ is uniformly equicontinuous. Indeed, by uniform continuity of $g$ , for any $\varepsilon>0$ we can find a neighborhood $U$ of $e\in G$ with $U =-U$ and $|g(s) - g(s')|\leq\varepsilon$ whenever $s-s'\in U$ holds and this gives
$$|F_i(t) - F_i (t')|\leq \varepsilon \|f\|_\infty$$
for all $i\in I$ whenever $t-t'\in U$ holds. As $(F_i)$ is uniformly equicontinuous, the set of $t$ where $(F_i(t))$ is a Cauchy-net is closed. This gives the desired statement.
\end{proof}


We can now prove the following result. The corresponding result for measures is \cite[Thm.~4.15]{LSS3}.

\begin{corollary}\label{cor:2} Assume that $G$ is $\sigma$ compact and has a countable dense set $D \in G$. Let $(A_n)_n$ be a van Hove sequence in $G$.  Then, for all  $f \in L^\infty(G), g \in \Cu(G)$ there exists a subsequence $\cB$ of $\cA$ such that $\lb f,g \rb_{\cB}$ exists.
\end{corollary}
\begin{remark} We note that the assumptions of the corollary are satisfied in second countable groups.
\end{remark}
\begin{proof} Let $t_m$ be an enumeration of $D$. Since
\[
\frac{1}{|A_n|} \int_{A_n} f(s) \overline{g(s-t_1)} \dd s
\]
is a bounded sequence in $\CC$, there exists some increasing sequence $k(n,1)$ of natural numbers such that
\[
\frac{1}{|A_{k(n,1)}|} \int_{A_{k(n,1)}} f(s) \overline{g(s-t_1)} \dd s
\]
converges.

Inductively, by the same argument, for each $j \geq 2$ we can construct a subsequence $k(n,j)$ of $k(n,j-1)$ such that
\[
\frac{1}{|A_{k(n,j)}|} \int_{A_{k(n,j)}} f(s) \overline{g(s-t_j)} \dd s
\]
converges.

A standard diagonalisation  argument gives a subsequence $\cB$ of $\cA$ such that
\[
\frac{1}{|B_n|} \int_{B_n} f(s) \overline{g(s-t_j)} \dd s
\]
converges for all $t \in D$. The claim follows now from Lemma~\ref{lem4}.
\end{proof}

\section{Construction of a unitary representation}\label{section-construction}
Let a  van Hove net $\cA$ be given.  We will be interested in subsets
$F$ of functions in $\Cu (G)$ with the property that  $\lb f,
g\rb_{\cA}$ exists for all $f,g\in F$. In this section we show that
any such set gives rise to a unitary representation on a suitably
defined Hilbert space. This will then be applied to sets $M$ of
translation bounded measures with the property that $\lb
\mu,\nu\rb_{\cA}$ exists for all $\mu,\nu\in M$.

We first characterize what
existence of all reflected Eberlein convolutions means for a set of
functions.

\begin{lemma}\label{Lem3}
Let $F \subseteq \Cu(G)$ and $\cA$ be a van Hove net. Then, the
following assertions are equivalent:
\begin{itemize}
\item[(i)] For all $f,g \in F$ the reflected Eberlein
convolution $\lb f,g \rb_{\cA}$ exists.

\item[(ii)] There exists set $B \subseteq \Cu(G)$ with the following properties:
\begin{itemize}
  \item{} $F \subseteq B$.
  \item{} $B$ is $G$-invariant.
  \item{} For all $f,g \in B$ the mean $M_{\cA}(f \overline{g})$ exists.
\end{itemize}
\end{itemize}
\end{lemma}
\begin{proof} (ii)$\Longrightarrow$(i):  For all $f,g \in F$  and $t\in G$ we have $f, \tau_tg \in B$ and
hence
\[
M(f \overline{\tau_tg})
\]
exists. This shows that $\lb f,g \rb$ exists.

(i)$\Longrightarrow$(ii): Define
\[
B_F:= \{ \tau_t f : f \in F, t \in G \} \,.
\]
We show that $B_F$ has the desired properties.

Let $f,g \in B_F$. Then, exist some $f_1,g_1 \in F$ and $t,s \in G$
so that
\[
f= \tau_t f_1 \,;\, g=\tau_s g_1 \,.
\]
Then, by invariance of the mean and the definition of the reflected
Eberlein convolution  we have
\[
M_{\cA}(f \bar{g})=M_{\cA}(\tau_tf_1
\overline{\tau_sg_1})=M_{\cA}(f_1 \overline{\tau_{s-t}g_1})=\lb
f_1,g_1 \rb_{\cA}(s-t)
\]
exists for all $f,g \in B_F$.

It is clear that $F \subseteq B_F$ and, by definition $B_F$, is $G$
invariant.
\end{proof}

\begin{remark}
We note that in the proof of Lemma~\ref{Lem3}, the set $B_F$ is the smallest $B$ satisfying the stated  conditions.
\end{remark}

Consider now a van Hove net  $\cA$ and let
 $F \subseteq \Cu(G)$ be given such that $ \lb f, g\rb_{\cA}$ exists
for all $f,g\in F$. We then define $H_F$ to be the linear span of
translates of elements from $F$, i.e.
\[
H_F:=\mbox{Span}\{ \tau_t f : t\in G, f \in F\} \,.
\]
From the assumption on existence of the reflected Eberlein
convolutions  we then find that
\[
\langle f, g\rangle := M_A (f\overline{g})
\]
exists for all $f,g\in H_F$. Clearly, $\langle \cdot,\cdot\rangle$
gives a semi-inner product on $H_F$. We denote by $\cH_F$ the
Hilbert space completion of $(H_F,\langle \cdot,\cdot\rangle)$.
By a slight abuse of notation we denote the inner product on $\cH_F$ by
$\langle \cdot,\cdot\rangle$ as well. Whenever $F$ consists just of
a single element $f$ we write $\cH_f$ instead of $\cH_{\{f\}}$.

Clearly, if $F$ and $F'$ are subsets of $\Cu(G)$ with $F\subseteq F'$ there is a canonical isometric embedding $\cH_F\hookrightarrow \cH_{F'}$ extending the embedding $F \hookrightarrow F', f \mapsto f$.

\begin{theorem}[The unitary representation $T^{F,\cA}$ induced by $F$]
\label{thm-unitary-rep-functions}
 Let $\cA$ be a van Hove net. Let $F \subseteq
\Cu(G)$ be such that $\lb f,g\rb_{\cA}$ exists for all $f,g\in F$.
Then, for every $t\in G$, the  translation operator $\tau_t : \Cu(G)
\to \Cu(G)$ induces a unitary operator
\[
T_t = T_t^{F,\cA} :\cH_F\longrightarrow \cH_F \,.
\]
 The family $T= T^{F,\cA} = (T_t^{F,\cA})_{t\in G}$ is a
representation of $G$, i.e. satisfies
\[
T_0 = \mbox{identity}  \mbox{ and }  T_{t+s} = T_t \circ T_s
\]
for all $t,s\in G$. This  representation is strongly continuous, i.e.  the map
\[
G\longrightarrow \cH_F, t\mapsto T_t f\,,
\]
is continuous for each $f\in \cH_F$.
\end{theorem}
\begin{proof} Set $H := H_F$ and  $\cH:=\cH_F$.

We show first that $\tau_t$ gives rise to a unique unitary operator:  Let $t \in G$ be arbitrary. Then, for all $f,g \in
H_F$ we have
\begin{equation}\label{eq1}
 \langle \tau_t  f, \tau_t g \rangle = M(\tau_tf
\overline{\tau_tg})= M(f \bar{g})= \langle  f, g \rangle
\end{equation}
 by
invariance of the mean. This means that $\tau_t: H \to H \subseteq
\cH$ is a bounded operator and hence it can uniquely be extended to
a bounded operator
\[
T_t: \cH \to \cH \,.
\]
Now, \eqref{eq1} and the denseness of $H$ in $\cH$ imply that $T_t$
is inner product preserving. Moreover,
$$\tau_{t} \circ \tau_{-t}=
\mbox{identity}$$ carries to $T_t$ and hence each $T_t$ is onto. This
shows that $T_t$ is an unitary operator for each $t$.

From $\tau_t\circ \tau_s = \tau_{t+s}$ and $\tau_0 =
\mbox{identity}$ we immediately infer that $T$ is a representation.

It remains to show the statement on strong continuity:  We first consider $f\in H$. Let $t\in G$ be fixed. Then, with respect to the semi-norm $\|\cdot\|$ induced by
$\langle ., . \rangle$ we have
\[
\| T_s f - T_t f\|^2= M_{\cA}( |\tau_s f - \tau_t f|^2) \leq \|\tau_s f - \tau_t f \|_{\infty}^2 =\|\tau_{s-t} f - f\|_\infty^2\,.
\]
Now, as any $f\in \Cu(G)$ is uniformly continuous, the term
$\|\tau_{s-t} f- f\|_\infty$ goes to zero for $s\to t$. This gives
the desired continuity for $f\in H$.  The case of general $f\in \cH$
can then be treated by denseness of $H$ in $\cH$ as follows:
Consider  $h \in \cH$. Let $\eps >0$. Then, there exists some $f \in
H$ such that $\| f-h \| < \eps$. By the above, there exists some
open set $0 \in U \subseteq G$ such that $t \in U $ implies $ \|
\tau_t f -f\| < \frac{\eps}{3}$. Then , for all $t \in U$ we have
\begin{align*}
  \| T_t h-h \| &\leq  \| T_t h-T_t f \|+ \| T_t f-f \|+ \|f-h\| =  \|  h- f \|+ \| \tau_t f-f \|+ \|f-h\|  \\
   &< \frac{\eps}{3}+\frac{\eps}{3}+\frac{\eps}{3} = \eps \,.
\end{align*}
This finishes the proof.
\end{proof}

By Stone Theorem and the previous theorem there exists  a (unique) map
\[
E : \mbox{Borel sets of $\widehat{G}$} \longrightarrow \mbox{
Projections on $\cH_F$}\]
with
\begin{itemize}
\item  $E(\emptyset) = 0$, (\textit{Non-degenerate});
\item  $E(\cup_n B_n) = \lim_N \oplus_{n=1}^N E(B_n)$ whenever $B_n$, $n\in\NN$,  are mutually disjoint measurable subsets of $\widehat{G}$, (\textit{$\sigma$-additive});
\item $\langle f,  T_t f\rangle   = \widecheck{\varrho_f}(t)$ for all $t\in G$.
 \end{itemize}

Here, the measure $\varrho_f$ is
defined by
\[
\varrho_f (B) = \langle f, E(B)f\rangle
\]
and called the \textit{spectral measure} of $f$. It is  uniquely
determined by
\[
 \langle f,T_t f \rangle =  \widecheck{\varrho_f}(t)
\]
for all $t\in G$. As $E$ takes values in the
projections we have
$$\varrho_f (B) =\langle f, E(B)f\rangle  = \langle E(B)f, E(B) f\rangle  = \|E(B)
f\|^2$$ for all $f\in \cH_F$ and any measurable $B\subset
\widehat{G}$. It is well-known that $E$ satisfies
\[
E(A)E(B) = E(A\cap B)
\]
for all measurable subsets $A,B\subset \widehat{G}$
as well as $E(\widehat{G}) = \mbox{Id}$. Indeed, the first equality follows easily from the additivity property of $E$. The second equality follows as $E(\widehat{G})$ is a projection with
\[
\|E(\widehat{G}) f\|^2 = \varrho_f (\widehat{G}) = \int_{\widehat{G}} 1 \dd
\varrho_f  = \langle f, T_0 f\rangle = \|f\|^2
\]
for all $f$.

\begin{remark} Some people prefer the version of the Stone Theorem where translation appear in the first (the linear) component, i.e.
\[
\langle T_t f,   f\rangle   = \widecheck{\varrho_f}(t) \,.
\]
To use this version, one would need to replace the operator $T_t$ induced by the left shift $\tau_t(f)=f(t-\cdot)$ with the operator $S_t=T_{-t}$ induced by the right shift on functions.
This change would induce a reflection of the spectral measure.
\end{remark}

We gather a few properties of  spectral measures next. This
clarifies in particular the relationship between a spectral measure
and the reflected Eberlein convolution.

 As usual we write $\mu \perp  \nu$ whenever $\mu$ and $\nu$ are positive measures on the same measurable space $X$ for which there exist measurable disjoint subsets $A,B$ of $X$  with $\mu(X\setminus A) = 0 = \nu(X\setminus B)$. The measures $\mu$ and $\nu$ are then called \textit{mutually singular}.

\begin{lemma}[Properties of spectral measures]\label{lemma2} Consider
 a van Hove net $\cA$ and  $F \subseteq \Cu(G)$  such that
$\lb f,g\rb_{\cA}$ exists for any $f,g\in F$ and let  $T =
T^{F,\cA}$ be the associated unitary representation. Then,
\begin{itemize}
\item[(a)] $\varrho_f = \varrho_{T_s f}$ for any $f\in\cH$ and $s\in G$.
\item[(b)] The spectral measure $\varrho_f$ of $f \in F$ satisfies
  \[
  \widecheck{\varrho_f}(t) = \lb f,f \rb_{\cA}(t) \qquad \mbox{ for all } t \in G \,.
  \]

  \item[(c)] Let $f,g \in F$ be given. Then $\lb f,g \rb_{\cA}=0$ if and only if $\cH_f \perp \cH_g$.
      \item[(d)] Let $f\in\cH$ be given. Then $f = E( A) f$ whenever $A\subset \widehat{G}$ satisfies  $\varrho_f (\widehat{G}\setminus A) =0$.

            \item[(e)] If $f,g \in \cH$ satisfy  $\varrho_f \perp \varrho_g$
                     then $\langle f,g\rangle = 0$.
                        \item[(f)] If $f,g \in F$ satisfy  $\varrho_f \perp \varrho_g$
                     then $\lb f, g\rb_{\cA}=0$.

\end{itemize}
\end{lemma}
\begin{proof}
\textbf{(a)} The characteristic feature of $\varrho_{T_s f}$ is that
\[
\int_{\widehat{G}}  \xi (t) \dd\varrho_{T_s f} ( \xi) = \langle T_s f, T_t T_s f\rangle
\]
holds for all $t\in G$. The characteristic feature of $\varrho_f$ is that
\[
\int_{\widehat{G}}  \xi (t)\dd\varrho_f(\xi) = \langle f, T_t f\rangle
\]
holds for all $t\in G$. As $T$ is unitary, we have $\langle T_s f, T_t T_s f\rangle = \langle f, T_t f\rangle$. Hence, the two spectral measures agree.

\textbf{(b)} For all $f \in F$ we have
\[
\widecheck{\varrho_f}(t)= \langle f, T_t f \rangle \stackrel{f \in
F}{=\joinrel =\joinrel =\joinrel =} M(f \overline{\tau_{t}f}) =\lb
f,f \rb_{\cA} (t) \,.
\]

\textbf{(c)}  For all $f,g \in F$ we have
\[
\lb f,g \rb (t)= M(f \overline{\tau_{t}g})
\stackrel{f,g \in F}{=\joinrel =\joinrel =\joinrel =}   \langle f,
T_t g \rangle \,.
\]
The claim follows immediately.

\textbf{(d)} We have $f = E(\widehat{G}) f = E(A) f+
E(\widehat{G}\setminus A) f$. Now, by assumption on $A$, we also
have
\[
\|E(\widehat{G}\setminus A)f \|^2 = \varrho_f (\widehat{G}\setminus
A) =0 \,.
\]
Put together this yields $f = E( A) f$.

\textbf{(e)}   Let $A,B$ be disjoint measurable subsets of $\widehat{G}$ with
\[
\varrho_f (\widehat{G}\setminus A) =0 = \varrho_g (\widehat{G}\setminus B)\,,
\]
Then, $f = E(A)f$ and $ g = E(B) g$ by (d). Hence, we can calculate
\[
\langle f, g\rangle = \langle E(A) f, E(B) g\rangle= \langle E(B) E(A) f,g\rangle = 0 \,,
\]
where we used   $E(A) E(B) = E(A\cap B) = E(\emptyset) =0$ to obtain the last equality.

\textbf{(f)} By definition we have  $ \lb f, g\rb_{\cA} (s) = \langle f, T_s g\rangle$. Now,  by (a) the spectral measure of $T_s g$ agrees with the spectral measure of $g$ and the desired statement follows from (e).
\end{proof}

The converse of (e) is not necessarily
true as shown in the next example:

\begin{example}[Converse of (e) does not hold]\label{ex1}  Let  $f$ be
\[
f(x)=
\left\{
\begin{array}{cc}
  -1 & \mbox{ if } x <-1 \\
  x & \mbox{ if } -1 \leq x \leq 1 \\
  1 & \mbox{ if } x>1
\end{array}
\right. \,.
\]
and let $g=1$ be the constant function $1$.
Set
\[
F:= \{ \tau_t f: t \in \RR\}\cup\{g\} \,,
\]
which is $G$ invariant. Then, by Lemma~\ref{lemma2} we have
\[
\widecheck{\varrho_f}(t) = \lb f,f \rb_{\cA}(t) = M(f \overline{\tau_t f}) =1 \qquad \mbox{ for all }  t \in \RR \,,
\]
with the last equality following from the fact that
\[
f(x) \overline{\tau_tf(x) } =1 \mbox{ for all }  x \notin [-|t|-1, |t|+1] \,.
\]
Also, for all $t \in G$ we have
\[
\widecheck{\varrho_g}(t) = \lb g,g \rb_{\cA}(t) = M(g
\overline{\tau_t g }) =1 \qquad  \mbox{ for all } t \in \RR \,.
\]
It follows that
\[
\sigma_{f}=\sigma_{g}= \delta_0 \,.
\]
On another hand, for all $t \in \RR$ we have
\[
\lb f, g \rb_{\cA}(t)= M(f \overline{\tau_{t}g})=M(f)=0 \,.
\]
\end{example}

\begin{remark}[Orthogonality of $\cH_f$ and $\cH_g$]
Let $F \subseteq \Cu(G)$ be given such that $\lb f,g\rb_{\cA}$ exists for all  $f,g \in F$. Then we have
\[
\cH_{\{ f,g \}} = \cH_f+ \cH_g \subseteq \cH_{F} \,.
\]
In many situations, $\cH_{f}$ and $\cH_{g}$ are not disjoint
subspaces, and then the sum above is not an orthogonal  sum. Indeed,
for example, when $g =\tau_t f$ for some $t \in G$ we trivially have
$\cH_f=\cH_g$. However,  Theorem~\ref{lemma2} gives the following
implications
\[
\bigl( \sigma_f \perp \sigma_{g} \bigr) \Longrightarrow \bigl( \lb f,g \rb_{\cA}=0 \bigr)\Longleftrightarrow \bigl( \cH_{f} \perp \cH_{g} \bigr) \,.
\]
Of course, if $\cH_f\perp \cH_g$ then  we trivially have $\cH_{\{ f,g \}} = \cH_f \oplus \cH_g$.
\end{remark}

We also note the following example.

\begin{example}[The space of Besicovitch-2-functions]\label{ex-bes} Let $F
=\widehat{G}$ be the dual group of $G$. Then, the linear span of $A$
is invariant under translations and for all $\xi,\eta\in
\widehat{G}$ the mean $M_\cA (\xi \overline{\eta})$ exists and
equals $0$ if $\xi \neq \eta$ and equals $1$ if $\xi = \eta$. Thus,
the elements of $\widehat{G}$ form an orthonormal basis of
$\cH_{\widehat{G}}$. The space $\cH_{\widehat{G}}$ is usually
denoted by $Bap^2_{\cA}(G)$, and is called the space of Besicovitch 2-almost periodic functions.
\end{example}

\medskip

We now turn to an application of the preceding considerations to
measures.

Let $M \subseteq \cM^\infty(G)$ be given  and let  $\cA$ be  van
Hove net and assume that  for all $\mu, \nu \in M$ the reflected
Eberlein convolution $\lb \mu, \nu \rb_{\cA}$ exists. Let us note
here in passing that this implies that the autocorrelation
$\gamma_{\mu}= \lb \mu , \mu \rb_{\cA}$ of each $\mu \in M$ exists
with respect to $\cA$. Define for $M$ the set
\[
F_M:=\{ \mu\ast\varphi : \mu \in M, \varphi \in \Cc(G) \}\subset \Cu
(G) \,.
\]
By the assumption on $M$ and Lemma \ref{lemma-relationsship} the
reflected Eberlein convolution $\lb f,g\rb_{\cA}$ exists then for
all $f,g\in F_M$. So, we can apply the preceding considerations. In
particular, $H_M:= H_{F_M}$ carries a  unique semi inner product
with
\[
\langle \mu \ast\varphi, \nu\ast \psi\rangle = M_{\cA} (\mu\ast
\varphi \; \overline{\nu \ast \psi})
\]
for all $\varphi,\psi\in\Cc
(G)$. The Hilbert space completion of $H_M$ will be denoted by
$\cH_M$.

If $\mu$ is a translation bounded measure such that
$\lb\mu,\mu\rb_{\cA}$ exists we  can apply the previous considerations to
 $M = \{ \mu \}$. In this case, we
simply write $\cH_\mu:= \cH_{\{ \mu \}}$. Note that then the set
\[
C:=\{\mu \ast \varphi : \varphi \in \Cc (G)\}
\]
is already a subspace of $\Cu (G)$ and  invariant under
translations. Hence, it is dense in $\cH_\mu$. Note also that for
$\mu $ in some subset $M$ of translation bounded measures such that
$\lb\mu,\nu\rb_{\cA}$ exists for all $\mu,\nu \in M$, the space
$\cH_{\mu}$ can be simply identified with the Hilbert subspace of
$\cH_{M}$ arising as the closure of $\{ \mu\ast\varphi : \varphi \in
\Cc(G) \}$.

\begin{theorem}\label{T2}
Let $M$ be a set of translation bounded measures  and $\cA$ a  van Hove net with the property that for all $\mu, \nu \in M$ the
reflected Eberlein convolution $\lb \mu, \nu \rb_{\cA}$ exists.
There exists a unique family $T :=T^M := T^{F_M,\cA}$ of unitary
operators on $\cH_M$ with
\[
T_t (\mu \ast \varphi) = \mu \ast (\tau_t \varphi)
\]
for all $\mu\in M$ and $\varphi \in\Cc (G)$.
For  all $\varphi \in
\Cc(G)$ and $\mu \in M$ we have
  \[
 \varrho_{\varphi\ast\mu} = \left| \widehat{\varphi} \right|^2 \reallywidehat{\gamma_{\mu}} \,.
  \]
\end{theorem}
\begin{proof} By construction and Theorem \ref{thm-unitary-rep-functions}, the family
$T$ is a unitary representation  with
\[
T_t (\mu \ast \varphi) =\tau_t (\mu \ast \varphi) = \mu \ast (\tau_t \varphi)
\]
for all $\mu\in M$ and $\varphi \in\Cc (G)$.

It remains to show the statement on the spectral measure: By
Lemma~\ref{lemma2}(b), Lemma
 \ref{lemma-relationsship}, the definition of $\gamma_\mu$ and Proposition \ref{char-ft-gamma}
  we have
  \begin{align*}
\widecheck{\varrho_{\mu\ast \varphi}}(t) &= \lb \mu\ast
\varphi,\mu\ast \varphi\rb_{\cA} (t)  =  \lb \mu,\mu \rb_{\cA} \ast \varphi \ast
\widetilde{\varphi}(t) \\
&= \gamma_\mu \ast \varphi \ast \widetilde{\varphi}(t)
=  \int_{\widehat{G}} \xi(t) |\widehat{\varphi}(\xi)|^2 \dd
\widehat{\gamma_\mu}(\xi) \,.
\end{align*} As the spectral
measure is uniquely determined by its inverse Fourier transform we
infer the desired statement.
\end{proof}

\begin{remark}[Off-diagonal spectral measures]
The result on the spectral measures in the previous
theorem can be seen as the `diagonal' part of a more general
statement and be complemented by an `off diagonal' part.
Specifically, consider the situation of the theorem. Then, by
polarisation we can write $\gamma_{\mu,\nu}:=\lb \mu,\nu\rb_{\cA}$
as a linear combination
\[
\gamma_{\mu,\nu}=\lb \mu,\nu\rb_{\cA} = \frac{1}{4} \sum_{k=0}^3 i^k \lb \mu  + i^k \nu, \mu + i^k
\nu\rb_{\cA}\,.
\]
Also we can  define the measure
\[
\widehat{ \gamma_{\mu,\nu}  } :=\frac{1}{4}\sum_{k=0}^3 i^k \widehat{\gamma_{\mu +
i^k \nu} }\,.
\]
Then, by construction and Lemma~\ref{lemma-relationsship}, we have
\[
\int_{\widehat{G}} \widehat{\varphi}(\xi) \overline{\widehat{\psi}(\xi) }
\dd \widehat{ \gamma_{\mu,\nu} }(\xi) = \gamma_{\mu,\nu}\ast \varphi \ast
\widetilde{\psi}(0) = \langle \mu \ast \varphi, \nu \ast \psi\rangle
\]
 holds for all $\varphi,\psi \in\Cc (G)$. Replacing $\psi$ by
$\tau_{-t} \psi$ we can argue as in the proof of Proposition
\ref{char-ft-gamma} to obtain
\[
\int_{\widehat{G}} \xi (t)\widehat{\varphi}(\xi) \overline{\widehat{\psi}(\xi)}
\dd \widehat{\gamma_{\mu,\nu}}(\xi)  = \langle \mu \ast\varphi, T_t (\nu\ast
\psi)\rangle
\]
for all $t\in G$. In this sense, $\widehat{\varphi}
\overline{\widehat{\psi}} \widehat{\gamma_{\mu,\nu}}$ can be
considered as spectral measure to the pair $(\mu \ast\varphi,
\nu\ast \psi)$.
\end{remark}

We can also characterize orthogonality of subspaces generated by
different measures.
\begin{theorem}[Orthogonality]\label{T3} Let $M \subseteq \cM^\infty(G)$ be any set and $\cA$ be any van Hove sequence with the property that for all $\mu, \nu \in M$ the reflected Eberlein convolution
$\lb \mu, \nu \rb_{\cA}$ exists. Then, for
$\mu,\nu\in\M$, the following are equivalent:
\begin{itemize}
\item[(i)] $\lb \mu, \nu \rb_{\cA}=0$.
\item[(ii)] $\langle \mu \ast \varphi,\nu\ast \psi\rangle = 0$ for
all $\varphi,\psi \in\Cc (G)$.
\item[(iii)] $\cH_{\mu\ast \varphi}\perp \cH_{\nu\ast \psi}$ for all
$\varphi,\psi \in \Cc (G)$.
\item[(iv)] $\cH_{\mu} \perp \cH_{\nu}$.
\end{itemize}
\end{theorem}
\begin{proof} (iv)$\Longrightarrow$(iii): This is clear as
$\cH_{\mu \ast\varphi}$ is a subspace of $\cH_\mu$ and $\cH_{\nu\ast
\psi}$ is a subspace of $\cH_\nu$.

(iii)$\Longrightarrow$(ii): This is obvious.

(ii)$\Longrightarrow$(iv): By definition, the set $\{\mu \ast
\varphi : \varphi \in\Cc (G)\}$ is dense in $\cH_\mu$ and the set
$\{\nu\ast \varphi : \varphi \in\Cc (G)\}$ is dense in $\cH_\nu$.
This easily gives the implication (ii)$\Longrightarrow$(iv).

(ii)$\Longleftrightarrow$(i) By Lemma \ref{lemma-relationsship} and
the definition of $\langle \cdot,\cdot\rangle$,  condition (ii) is
equivalent to $\lb \mu,\nu\rb_{\cA} \ast \varphi \ast
\widetilde{\psi} (0) =0$ for all $\varphi,\psi \in\Cc (G)$. This in
turn is equivalent to  (i).
\end{proof}

\section{Diffraction as spectral measure}\label{section-main}

Starting from a measure $\mu$ with autocorrelation $\gamma$ we can
now construct explicitly a unitary representation of $G$ on the
Hilbert space $\cH_{\mu}$ whose spectrum is exactly
$\widehat{\gamma}$. We then go on and use this to characterize pure point diffraction  and then  discuss orthogonality features
with respect to this Hilbert space.

\begin{theorem}[Diffraction as a spectral
measure]\label{diff-spec} Let $\mu$ be a translation
bounded measure and let $\cA$ be a van Hove net such that the autocorrelation $\gamma_{\mu}$ of $\mu$ exists with
respect to $\cA$. Then, there exists a unique unitary map
\[
\Theta : \cH_\mu \longrightarrow L^2
(\widehat{G},\widehat{\gamma_\mu})
\]
such that $\Theta(\mu\ast \varphi)=\widehat{\varphi}$ for all $\varphi \in\Cc (G)$. This map satisfies
\[
\Theta \circ T_t = Z_t \circ \Theta \,,
\]
for all $t\in G$, where $Z_t$ is the operator on $L^2 (\widehat{G},\widehat{\gamma_\mu})$  defined by
\[
(Z_t f)(\xi) = \xi(t) f (\xi) \,.
\]
\end{theorem}
\begin{proof} By definition, $H_{\mu}=\{\mu \ast \varphi : \varphi \in\Cc
(G)\}$ is a dense subspace of $\cH_\mu$.

Uniqueness of $\Theta$ follows immediately as $H_{\mu}$ is dense. As for
existence we note that for $\varphi \in\Cc (G)$ we have by Theorem
\ref{T2}
\[
\|\mu\ast \varphi\|^2 = \langle \mu\ast \varphi, T_0 \mu\ast \varphi
\rangle   = \int_{\widehat{G}} \dd \sigma_{\mu \ast \varphi}(\xi) =
\int_{\widehat{G}} |\widehat{\varphi}(\xi)|^2
\dd\widehat{\gamma_\mu}(\xi)\,.
\]
Combined with the
denseness of $H_{\mu}$ this shows that there exists a unique isometric
$\Theta$ mapping $\mu\ast \varphi $ to $\widehat{\varphi}$ for
$\varphi \in\Cc (G)$. Now, the set of
$\widehat{\varphi}$, $\varphi \in\Cc (G)$, is dense in $L^2
(\widehat{G},\widehat{\gamma_\mu})$ \cite[Prop. 2.20]{ARMA1}. So, $\Theta$ is an isometry
with dense range and, hence, a unitary map.

The equality $\Theta\circ T_t = Z_t \circ \Theta$ is clear on  $H_{\mu}$ and
then follows by denseness on $\cH_\mu$.
\end{proof}

\begin{remark}[Understanding $\Theta$]
\begin{itemize}
\item[(a)] The   map $\Theta$ diagonalizes the action $T$ in the sense that
it transforms it into multiplication operators.

\item[(b)] Let $\gamma$ be a positive definite measure on $G$. Then,
$\gamma$ induces semi-inner product on $\Cc (G)$ by
\[
\langle \varphi,\psi\rangle_\gamma :=\gamma\ast \varphi \ast
\widetilde{\psi} (0) \,.
\]
We can then form the Hilbert space completion
$(\cH_\gamma,\langle\cdot,\cdot\rangle)$ of $\Cc (G)$ equipped with
$\langle \cdot,\cdot\rangle_\gamma$. This space can naturally be
identified with  $L^2 (\widehat{G}, \widehat{\gamma})$ via the
unique unitary extension $U_\gamma$  of the map
\[
\Cc (G)\longrightarrow L^2 (\widehat{G},\widehat{\gamma}),
\varphi \mapsto \widehat{\varphi}\,.
\]
Indeed, a proof can be given by
mimicking the arguments given in the proof of the previous theorem.

Furthermore, if $\gamma = \gamma_\mu$ is the autocorrelation of a
translation bounded measure $\mu$, we can naturally identify
$\cH_{\gamma_\mu}$ with $\cH_\mu$ via the unique unitary extension
$V_\mu$ of the map
\[
\Cc (G)\longrightarrow \Cu(G),\varphi \mapsto \mu \ast\varphi \,.
\]
Then, $\Theta$ is just given as the composition  $U_{\gamma_\mu}
\circ (V_\mu)^{-1}$.
\end{itemize}
\end{remark}

\begin{remark}[Relating $\cH_\mu$ to the literature]\label{rem-hull}
\begin{itemize}
\item[(a)] In this part of the remark we relate the above
approach to the Hilbert-space $H_\mu$  to the approach via processes
given in \cite{LM,DM}:
 Whenever $\mu$ is a translation bounded measure with
autocorrelation $\gamma$ we can define
\[
N : \Cc (G)\longrightarrow \cH_\mu, \varphi \mapsto
\mu\ast\varphi
\]
and $N' : \Cc (G)\longrightarrow L^2 (\widehat{G},\widehat{\gamma}),
\varphi \mapsto \widehat{\varphi}$. Then, $(N,\cH_\mu,T)$ is a
process in the sense of \cite{LM} by Theorem \ref{T2} and so is
$(N',L^2(\widehat{G},\widehat{\gamma}))$ (by \cite{LM}). The previous
theorem can be understood as saying that these two processes are
equivalent (spatially isomorphic).

\item[(b)] The Hilbert space $\cH_{\mu}$ can be alternately be understood the following way:

Let
\[
\Omega_\mu:=\overline{\{\tau_t \mu: t\in G\}}
\]
be the hull of $\mu$. Here, the closure is taken in the vague topology. Let $m$ be an
invariant measure of  on $\Omega_\mu$ and define
\[
N : \Cc (G)\longrightarrow L^2 (\Omega_\mu,m), N_\varphi
(\omega):=\omega \ast \varphi (0)\,.
\]
Let $\mathcal{S}(m)$ be the
closure of the range of $N$ in  $L^2 (\Omega_\mu,m)$. Assume now
that $\mu$ is generic for $m$, where generic means that
\[
\frac{1}{|A_i|} \int_{A_i} f(\tau_t \mu) \dd t \to \int_{\Omega_\mu} f (\omega) \dd m(\omega)
\]
holds for all continuous $f$ on $\Omega_\mu$. Then, the map $\mu
\ast \varphi \to N_\varphi$ can be extended uniquely to a unitary
map between $\cH_\mu$ and $\mathcal{S}(m)$.

\item[(c)] Part (b) of this remark raises the question whether  any $\mu$
that admits an autocorrelation along $(A_i)$  also admits a measure $m$ on its
hull such that $\mu$ is generic with respect to this $m$ and $\mathcal{S}(m)$ can be identified with $\mathcal{H}_\mu$ as in (b).  If $G$ is second countable and $(A_n)$ is a van Hove sequence, this can  be shown to hold after one replaces $(A_n)$ by a  suitable subsequence $(B_n)$.  Indeed, by second countability of $G$ the hull $\Omega_\mu$ is metrisable and compact. Hence, $C(\Omega_\mu)$ is second countable.  Therefore, we can chose  a subsequence $(B_n)$ of $(A_n)$ such that
$$\frac{1}{|B_n|} \int_{B_n} f(\tau_t\mu) dt$$ converges for all $f\in C(\Omega_\mu)$. With $m$ defined by
$$m(f):=\lim_n \frac{1}{|B_n|} \int_{B_n} f(\tau_t\mu) dt$$
 we then have found a suitable measure.
\end{itemize}
\end{remark}

In particular, Remark~\ref{rem-hull} (c) has the following important consequence (see \cite{BL} for definitions).

\begin{proposition}[Diffractions of measures are diffractions of dynamical systems] Let $G$ be a second countable LCAG, let $\mu$ be a translation bounded measure on $G$ and let $\cA$ be a van Hove sequence
such that the autocorrelation $\gamma$ of $\mu$ exists with respect to $\cA$. Then, there exists a $G$-invariant measure $m$ on $\Omega_{\mu}$ such that $\gamma$ is the autocorrelation measure of $(\Omega_{\mu}, G, m)$.

In particular, $\widehat{\gamma}$ is the diffraction measure of  $(\Omega_{\mu}, G, m)$.
\end{proposition}
\begin{proof}
Let $B_n$ and $m$ be as in Remark~\ref{rem-hull} (c). Since, $\mu$ is generic for $m$, \cite[Lemma 7]{BL} gives that $\gamma$ is the autocorrelation of $(\Omega_{\mu}, G, m)$.
\end{proof}

\begin{remark} For ergodic measures the converse also holds.

Let $G$ be any second countable LCAG, let $\mu$ be a translation bounded measure on $G$, let $m$ be a $G$-invariant probability measure and let $\gamma$ be the autocorrelation of
$(\Omega_{\mu}, G,m)$.

If $\omega \in \Omega_{\mu}$ and $\cA$ is a van Hove sequence such that $\omega$ is generic for $m$ with respect to $\cA$, then $\gamma$ is the autocorrelation of $\omega$.

In particular, when $m$ is ergodic, and $\cA$ is a van Hove sequence along which Birkhoff ergodic theorem holds, then there exists elements $\omega \in \Omega_{\mu}$ such that
$\gamma$ is the autocorrelation of $\omega$ with respect to $\cA$ (compare \cite{BL}).
\end{remark}

The preceding considerations allow us to (re)prove various characterisations of  pure point diffraction.  Recall that a  subset $S$ of $G$ is \textit{relatively dense} if there exists a compact $K\subset G$ with $S+ K = G$ and that a
bounded  continuous function $f : G\longrightarrow \CC$ is \textit{Bohr-almost periodic} if for any $\varepsilon >0$ the set
$$\{t\in G : \|f- \tau_t f\|_\infty \leq \varepsilon\}$$
is relatively dense.

\begin{corollary}[Characterization of pure point diffraction] Let $\mu$ be a translation bounded measure with autocorrelation $\gamma$ and associated unitary representation $T$ on $(\cH_{\mu}, \langle \, , \, \rangle)$. Then, the following assertions are equivalent:
\begin{itemize}
\item[(i)] The unitary representation $T$ has pure point spectrum.

\item[(ii)] The diffraction measure $\widehat{\gamma}$ is a pure point measure.

\item[(iii)]  The autocorrelation is $\gamma$ is strongly almost periodic, i.e. for any $\varphi \in C_c (G)$ the function $\gamma \ast \varphi$ is Bohr almost periodic.

\item[(iv)] For any $\varphi \in C_c (G)$ the function $t\mapsto \langle T_t \mu\ast \varphi, \mu\ast \varphi\rangle$ is Bohr-almost periodic.

\item[(v)] For any $\varphi  \in C_c (G)$ and any $\varepsilon >0$  the set $$\{t\in G: |\langle T_t \mu\ast \varphi,\mu\ast \varphi\rangle -\langle \mu\ast \varphi,\mu\ast \varphi\rangle|\leq \varepsilon\}$$
 is relatively dense.
\item[(vi)]  The measure $\mu$ is  mean almost periodic, i.e. for each $\varepsilon>0$ the set
$$\{t\in G : \|T_t \mu\ast \varphi -\mu \ast \varphi\|\leq \varepsilon\}
$$
is relatively dense.
\end{itemize}
\end{corollary}
\begin{proof} By the Thm.~\ref{diff-spec}, $T$ has pure point spectrum if and only if $Z_\cdot$ has pure point spectrum. This is turn is easily seen to be equivalent to $\widehat{\gamma}$ being a pure point measure. In this way,
the equivalence between (i) and (ii) follows  from Thm.~\ref{diff-spec}. Now, clearly (ii) is equivalent to each spectral measure
$$ \varrho_{\mu\ast \varphi} =|\widehat{\varphi}|^2 \widehat{\gamma}$$ being pure point.  This, in turn is just equivalent to (iv) by Wiener lemma (see e.g. \cite{LS-chemnitz,MoSt})
Now, the equivalence between (iv), (v) and (vi) is  standard for unitary representations (see e.g. \cite{LS-chemnitz} as well).
It remains to show the equivalence between (iii) and (iv). Now, by what we have shown
$$\gamma \ast \varphi\ast \widetilde{\varphi}(t) =\langle T_t \mu\ast \varphi,\mu\ast \varphi\rangle$$
holds for all $\varphi \in C_c (G)$. Hence, (iv) is equivalent to
almost periodicity of
$$t\mapsto \gamma\ast \varphi \ast \widetilde{\varphi}$$
for all $\varphi \in C_c (G)$. By polarisation this  is equivalent to Bohr-almost periodicity of
$$t\mapsto \gamma\ast \varphi \ast \widetilde{\psi}$$
for all $\varphi,\psi\in C_c (G)$.  This, in turn, can easily be seen to be equivalent to (iii).
\end{proof}

\begin{remark}
The equivalence between (ii) and (iii) has first been shown by Baake / Moody \cite{BM-Crelle}. The equivalence between (ii) and (vi)  is contained in recent work of the authors with Spindeler \cite{LSS}. Our proof of these equivalences as well as the other equivalences are new (as they are based on the  unitary representation $T$ which has not been considered before).
\end{remark}

\section{An orthogonality result}\label{section-main2}

We now turn to our main result on orthogonality.

\begin{theorem}[Orthogonality with respect to the twisted Eberlein
convolution]\label{thm:main} Let $\mu, \nu$ be translation bounded
measures and $\cA$ a van Hove net such that the autocorrelations
$\gamma_{\mu}$ of $\mu$ and $\gamma_{\nu}$ of $\nu$ exist with
respect to $\cA$. If $\reallywidehat{\gamma_{\mu}} \perp
\reallywidehat{\gamma_{\nu}}$ holds  then $\lb \mu, \nu \rb_{\cA}$ exists
and satisfies
$\lb \mu, \nu \rb_{\cA}=0.$
\end{theorem}
\begin{proof}
By the compactness statement in (a) of Lemma  \ref{thm:rebe props} it suffices to show
$\lb \mu,\nu\rb_{\cB} =0$  whenever  $\cB$ is a subnet of $\cA$ such that $\lb
\mu, \nu \rb_{\cB}$ exist.
Therefore, without loss of generality we can assume that $\lb \mu,
\nu \rb_{\cA}$ exists.

Let $M:= \{ \mu, \nu \}$. By Theorem~\ref{T2}, we have an unitary
representation of $G$ on the space $\cH_{\mu,\nu}:= \cH_{M}$ and,
for all $\varphi, \psi \in \Cc(G)$ we have
\[
\varrho_{\varphi\ast\mu} = \left| \widehat{\varphi} \right|^2 \reallywidehat{\gamma_{\mu}} \mbox{ and }
 \varrho_{\psi*\nu} = \left| \widehat{\psi} \right|^2 \reallywidehat{\gamma_{\nu}}\,.
 \]

Since $\reallywidehat{\gamma_{\mu}} \perp
\reallywidehat{\gamma_{\nu}}$ we get $\varrho_{\mu\ast \varphi}
\perp \varrho_{\nu\ast \psi}$. Then, (e) of  Lemma \ref{lemma2}
implies
\[
0 = \langle \mu \ast \varphi, \nu\ast \psi\rangle
\]
for all $\varphi,\psi \in\Cc (G)$.
Given this the  desired statement now follows from Theorem~\ref{T3}.
\end{proof}

\begin{corollary}  Let $\mu, \nu$ be translation bounded
measures and $\cA$ a van Hove net such that the autocorrelations
$\gamma_{\mu}$ of $\mu$ and $\gamma_{\nu}$ of $\nu$ exist with
respect to $\cA$. If $\reallywidehat{\gamma_{\mu}} \perp
\reallywidehat{\gamma_{\nu}}$ then
for all $a,b \in \CC$ the measure $a\mu+b\nu$ has diffraction
\[
\reallywidehat{\gamma_{a\mu+b\nu}}= |a|^2 \reallywidehat{\gamma_{\mu}}+ |b|^2 \reallywidehat{\gamma_{\nu}} \,.
\]
\end{corollary}
\begin{proof} This is proven exactly like Pythagoras' theorem in inner
product spaces. Indeed, by the preceding theorem we find for the autocorrelation measures the following:
\begin{align*}
 \gamma_{a\mu+b\nu} &= \lb a\mu+b\nu,a\mu+b\nu \rb_{\cA}
   =|a|^2 \lb \mu, \mu \rb_{\cA} +a\bar{b} \lb \mu, \nu \rb_{\cA} +b\bar{a} \lb \nu, \mu \rb_{\cA} + |b|^2\lb \nu, \nu \rb_{\cA} \\
   &=|a|^2 \gamma_{\mu}+ |b|^2 \gamma_{\nu}\,.
   \end{align*}
Taking the Fourier transforms we get
the claim.
\end{proof}

\begin{remark}
Validity of $\lb \mu, \nu \rb_{\cA}=0$ does not necessarily imply
that $\reallywidehat{\gamma_{\mu}}$ and
$\reallywidehat{\gamma_{\nu}}$ are mutually singular.  Indeed,
similarly to Example~\ref{ex1} we can show that the measures
\[
\mu=\delta_{\ZZ}\, ;\, \nu=\sum_{m \in \ZZ} \mbox{sgn}(m) \delta_m
\]
satisfy with respect to $A_n=[-n,n]$:
\[
\lb \mu, \nu \rb_{\cA}=0 \mbox{ and } \widehat{\gamma_{\mu}} = \widehat{\gamma_{\nu}}=\delta_{\ZZ} \,.
\]
\end{remark}

\section{Application: The point part of the diffraction and Bombieri--Taylor type results}
In this section we  discuss consequences of the main orthogonality result to diffraction theory. This sheds a new light on what is sometimes known as Bombieri--Taylor conjecture (compare remark at the end of this section).

\begin{proposition}\label{prop2}
Let $\mu$ be a translation bounded measure and let $\cA$ be a van
Hove net. Assume that the autocorrelation $\gamma_{\mu}$ exists
with respect to $\cA$. Then, for all $\chi \in \widehat{G}$ for
which $\reallywidehat{\gamma_{\mu}}(\{ \chi \}) = 0$ the
Fourier--Bohr coefficient $a_{\chi}^\cA(\mu)$ exists and satisfies
\[
a_{\chi}^\cA(\mu)=0 \,.
\]
\end{proposition}
\begin{proof}
Let $\nu:=\chi \theta_{G}$. Then, $\gamma_{\nu}$ exists with respect to $\cA$ and
\[
\reallywidehat{\gamma_{\nu}}=\delta_{\chi} \,.
\]
Therefore, by Theorem~\ref{thm:main} the reflected Eberlein
convolution  $\lb \mu, \chi \rb_{\cA}$ exists and is zero. The claim
follows now from Proposition \ref{prop-FB}.
\end{proof}

\begin{remark}[Converse of previous proposition fails] The converse of the previous proposition is not true. Indeed, let
\[
\mu:= \sum_{m \in \ZZ} \mbox{sgn}(m) \delta_m \,.
\]
and let $A_n=[-n,n]$.
Then, it is clear that
\[
a_0^{\cA}(\mu)=0 \,.
\]
On another hand, the autocorrelation $\gamma_{\mu}$ exists with respect to $A_n$ and
\[
\gamma_{\mu}= \delta_{\ZZ} \,.
\]
It follows that
\[
\reallywidehat{\gamma_{\mu}}(\{0\})=1 \,.
\]
\end{remark}

The proposition can be used to study existence of the Fourier-Bohr coefficients.  This is carried out next.

Recall first that a set $A$ is called locally countable if $A \cap K$ is countable for all compact sets $K \subseteq G$. If $G$ is $\sigma$-compact, then $A \subseteq G$ is locally countable if and only if it is countable.

\begin{proposition}\label{cor:FB}  Let $\mu \in \cM^{\infty}(G)$ and $\cA$ a van Hove net such that the autocorrelation $\gamma$ of $\mu$ exists with respect to $\cA$. Then, there exists some locally countable set $A \subseteq \widehat{G}$ such that for all $\chi \notin A$ the Fourier--Bohr coefficient $a_{\chi}^{\cA}(\mu)$ exists and is zero. In particular, the set of $\chi \in \widehat{G}$ for which the following limit does not exist is locally countable:
\[
\lim_i \frac{1}{|A_i|} \int_{A_i} \overline{\chi(t)} \dd \mu(t) \,.
\]
\end{proposition}
\begin{proof}  Since $\widehat{\gamma}$ is a measure, the set
\[
A = \{ \chi : \widehat{\gamma}(\{ \chi \}) \neq  0\}
\]
is locally countable. The first claim now  follows from Proposition~\ref{prop2}. The 'In particular' statement now is an immediate consequence.
\end{proof}

\begin{corollary}[Existence of Fourier-Bohr-coefficients along a subsequence]\label{cor:FBexist}  Assume that $G$ is second countable. Let $\mu \in \cM^{\infty}(G)$ and $\cA$ a van Hove sequence. Then, there exists a subsequence $\cB$ of $\cA$ along which
the autocorrelation $\gamma$ and all Fourier--Bohr coefficients $a_{\chi}^{\cB}(\mu)$ exist.
\end{corollary}
\begin{proof} Pick first a sub sequence $\cA'=\{ A_{k_n} \}$ of $\cA$ along which the autocorrelation $\gamma$ exists. Since $\widehat{G}$ is $\sigma$-compact, by  Proposition~\ref{cor:FB}
the Fourier--Bohr coefficients exist outside a countable set $D$ of characters.

Let $\chi_1, \chi_2, \ldots, \chi_n, \ldots$ be an enumeration of $D$. By boundedness, there exists a subsequence $k(1,n)$ of $k_n$ such that the following limit exists.
\[
\lim_n \frac{1}{|A_{k(1,n)}|} \int_{A_{k(1,n)}} \overline{\chi_1(t)} \dd \mu(t)
\]
Inductively we can now construct a subsequence $k(m+1,n)$ of $k(m,n)$ along which
\[
\lim_n \frac{1}{|A_{k(n+1,n)}|} \int_{A_{k(m+1,n)}} \overline{\chi_{m+1}(t)} \dd \mu(t)
\]
exists.

A standard diagonalisation argument proves the claim.
\end{proof}

We now turn to another consequence  (or rather a reformulation) of Proposition \ref{prop2}.

\begin{corollary}[Bombieri--Taylor type result] Let $\mu$ be a translation bounded measure and assume that the autocorrelation $\gamma_{\mu}$ exists with respect to $\cA$.
If
\[
\lim_{i} \frac{1}{|A_i|} \int_{A_i} \overline{\chi(t)} \dd \mu(t) =0
\]
does not hold,  then
\[
\widehat{\gamma_{\mu}}(\{ \chi \}) \neq 0\,.
\]
\end{corollary}
\begin{proof}
Assume by contradiction that $\widehat{\gamma_{\mu}}(\{ \chi \}) =
0$. Then, by Proposition \ref{prop2} we have
\[
\lim_{n} \frac{1}{|A_n|} \int_{A_n} \overline{\chi(t)} \dd \mu(t) =0
\,.
\]
\end{proof}

An immediate consequence of the preceding corollary is the following.

\begin{corollary}\label{cor1}
Let $\mu$ be a translation bounded measure and $\cA$ a van Hove
sequence. Assume that the autocorrelation $\gamma_{\mu}$ of $\mu$
exists with respect to $\cA$ and that $\reallywidehat{\gamma_{\mu}}$
is a continuous measure. Then, all Fourier--Bohr coefficients
$a_\chi(\mu)$ exist and satisfy
\[
a_{\chi}^{\cA}(\mu)=0 \qquad \mbox{ for all } \chi \in \widehat{G} \,.
\]
\qed
\end{corollary}

The preceding results have shown how non-vanishing of the
Fourier--Bohr coefficient implies non-vanishing of the pure point diffraction
component. We now turn to converse type of statements. These converses
tend to need some extra uniformity assumption.

\begin{lemma}Let $G$ be second countable and let $\cA$ be a van Hove sequence. Let $\mu$ be a translation bounded measure and assume that the autocorrelation $\gamma_{\mu}$ exists with respect to $\cA$.
If
\[
\widehat{\gamma_{\mu}}(\{ \chi \}) \neq 0 \,.
\]
then there exists some $t_n \in G$ such that
\[
\lim_{n} \frac{1}{|A_n|} \int_{t_n+A_n} \overline{\chi(t)} \dd \mu(t) \neq 0 \,.
\]
\end{lemma}
\begin{proof}
Assume by contradiction that $t_n \in G$ such that
\[
\lim_{n} \frac{1}{|A_n|} \int_{t_n+A_n} \overline{\chi(t)} \dd \mu(t) = 0 \,.
\]
Then, by \cite{LSS}, the Fourier--Bohr coefficient $a_{\chi}(\mu)$ exists uniformly. Therefore, by \cite{Str}, the CPP holds, that is
\[
\widehat{\gamma_{\mu}}(\{ \chi \})= |a_{\chi}(\mu)|^2=0 \,.
\]
\end{proof}

Combining all results in this section, we get the following
consequence.

\begin{corollary}\label{cor11} Let $G$ be second countable, let $\mu$ be a translation bounded measure on $G$  and let $\cA$ be a fixed van Hove sequence.

Let $\Gamma$ be the set of all vague cluster points of the sequences
\[
\frac{1}{|A_n|} (\mu|_{t_n+A_n}\ast\widetilde{\mu|_{t_n+A_n}})
\]
for all choices of $t_n \in G$.

Let $A$ be the set of all vague cluster points of the sequences
\[
\frac{1}{|A_n|} \int_{t_n+A_n} \overline{\chi(t)} \dd \mu(t)
\]
for all choices of $t_n \in G$.

Then, for $\chi \in \widehat{G}$, the following are equivalent:
\begin{itemize}
  \item[(a)] $\widehat{\gamma}(\{ \chi \})=0$ for all $\gamma \in \Gamma$.
  \item[(b)] $a_\chi=0$ for all $a_\chi \in A$.
  \item[(c)] The Fourier--Bohr coefficient $a_{\chi}(\mu)$ exists uniformly and is zero.
\end{itemize}

\end{corollary}

\begin{remark}  Let $\varOmega_\mu$ be the hull of $\mu$ (compare (b) of  Remark \ref{rem-hull}).
The set $\Gamma$ represents the set of all autocorrelations of
elements $\omega \in \varOmega_\mu$ calculated along subsequences of
$A_n$. Same way, $A$ represents the set of all Fourier--Bohr
coefficients of $\omega \in \varOmega_\mu$ calculated along subsequences
of $A_n$.

Therefore, Corollary~\ref{cor11} says that, for a fixed $\chi \in \widehat{G}$, the following are equivalent:
\begin{itemize}
  \item[(a)] $\chi$ is not a Bragg peak for any $\omega \in \varOmega_\mu$.
  \item[(b)] $a_\chi^{\cA}(\omega)=0$ for all $\omega \in \varOmega_\mu$.
  \item[(c)] $a_\chi^{\cA}(\omega)=0$ uniformly for all $\omega \in \varOmega_\mu$.
\end{itemize}
\end{remark}

\medskip

We also have the following characterisation of absence of point
spectrum in the diffraction.

\begin{proposition} Let $G$ be second countable, let $\mu \in \cM^\infty(G)$ be a measure and $\cA= \{A_n \}$ be a van Hove sequence such that for all sequences $t_n \in G$, every cluster points of
\[
\frac{1}{|A_n|} \mu|_{t_n+A_n}\ast\widetilde{\mu|_{t_n+A_n}}
\]
has continuous Fourier transform. Then,
\begin{itemize}
  \item[(a)] The Fourier--Bohr coefficients $a_{\chi}(\mu)$ exist uniformly and $a_{\chi}(\mu)=0$.
  \item[(b)] If $\eta$ is the autocorrelation of $\mu$ with respect to some van Hove sequence, then $\widehat{\eta}$ is continuous.
\end{itemize}
\end{proposition}
\begin{proof}
\textbf{(a)}  Fix $t_n \in G$. Since
\[
\frac{1}{|A_n|} \int_{t_n+A_n} \overline{\chi(t)} \dd \mu(t)
\]
is bounded, it converges to $0$ if and only if $0$ is the only cluster point of this sequence.

Pick a subsequence $(k_n)$ such that
\[
c= \lim_n \frac{1}{|A_{k_n}|} \int_{t_{k_n}+A_{k_n}} \overline{\chi(t)} \dd \mu(t) \,,
\]
exists. Now, by translation boundedness, there exists some subsequence $l_n$ of $k_n$ such that the autocorrelation $\gamma$ of $\mu$ exists along $t_{l_n}+A_{l_n}$.
By the condition in the statement, $\widehat{\gamma}$ is continuous, and hence, by Corollary~\ref{cor1} we get $c=0$.

This shows that the Fourier--Bohr coefficient $a_\chi(\mu)$ exists and is $0$ for all translates of $A_n$, and hence it exists uniformly by \cite{LSS}.

\textbf{(b)} Since the Fourier--Bohr coefficients exist uniformly and are 0, they exist uniformly and are 0 for all choices of van Hove sequence \cite{LSS}.
Therefore, the CPP holds for $\eta$ by \cite{Len,Str} (compare \cite{HOF,HOF1} for $G=\RR^d$). This shows that $\widehat{\eta}$ is continuous.
\end{proof}

\begin{remark}[Bombieri--Taylor] A substantial part of the recent interest in diffraction
theory comes from the discovery of quasicrystals by Shechtman
\cite{She}, which was later honored with a nobel prize. The
characteristic feature of quasicrystals is pure point diffraction
together with symmetries which exclude periods. Accordingly, a key
point in the theoretical investigation is the study of pure point
diffraction. Here, a particular issue is to compute the atoms of
$\widehat{\gamma_\mu}$ (assuming that $\widehat{\gamma_\mu}$ is a
pure point measure). From the very beginning the idea was that the
atoms of $\widehat{\gamma_\mu}$ are exactly those $\xi$ where the
Fourier--Bohr coefficient does not vanish. Indeed, this is assumed
in large parts of the physics literature.  In the mathematics
literature this is sometimes known as Bombieri--Taylor conjecture
(after  \cite{BT1,BT2} where this was assumed without any
reasoning). An even stronger condition found in many places  is that
$\widehat{\gamma_\mu}(\{\xi\} = |a_\xi|^2$ holds (see for example
\cite{Oli,Lag,Len,LSS,LSS2}). This condition is known as consistent
phase property. Our treatment above provides the most general
treatment of Bombieri--Taylor conjecture so far. In particular, it is
not restricted to examples coming from (uniquely) ergodic dynamical
systems.  However,  it does not prove the  consistent phase
property.
\end{remark}

\section{Application: Existence of the refined Eberlein decomposition of an arbitrary measure}

One  open problem in diffraction theory is the following question:

\begin{question} For which $\omega \in \cM^\infty(G)$ and van
Hove net $\cA$ such that $\gamma$ exists (with respect to
$\cA$) can we find some measures $\omega_{dpp}, \omega_{dac},
\omega_{dsc} \in \cM^\infty(G)$ with the following properties:
\begin{itemize}
  \item[(a)] $\omega=\omega_{dpp} + \omega_{dac}+ \omega_{dsc}$.
  \item[(b)] The autocorrelations $\eta_{dpp},\eta_{dsc}, \eta_{dac}$ of $\omega_{dpp}, \omega_{dac}, \omega_{dsc}$ exist with respect to $\cA$ and
  \[
    \widehat{\eta_{dpp}}=\left( \widehat{\gamma} \right)_{pp}  \mbox{ and }  \widehat{\eta_{dac}}=\left( \widehat{\gamma} \right)_{ac} \mbox{ and }
    \widehat{\eta_{dsc}}=\left( \widehat{\gamma} \right)_{sc} \,.
    \]
  \end{itemize}
\end{question}
We will refer to any such decomposition as  \emph{refined Eberlein
decomposition of $\omega$}.

We are not able to answer this question here but we note that our
main result (Theorem \ref{thm:main})  has the following consequence:

\begin{corollary}\label{cor3}
Let $\mu, \nu,\omega$ be translation bounded measures and $\cA$ a
van Hove net such that the autocorrelations $\gamma_{\mu}$ of
$\mu$, $\gamma_{\nu}$ of $\nu$ and $\gamma_{\omega}$ of $\omega$
exist with respect to $\cA$. If $\reallywidehat{\gamma_{\mu}}$ is
pure point, $\reallywidehat{\gamma_{\nu}}$ is absolutely continuous
and $\reallywidehat{\gamma_{\omega}}$ of $\omega$ is singular
continuous, then
\[
\lb \mu, \nu \rb_{\cA}=\lb \mu, \omega \rb_{\cA}=\lb \nu, \omega
\rb_{\cA}=0 \,.
\]
In particular, for all $a,b,c \in \CC$ the autocorrelation of
$a\mu+b\nu+c\omega$ exists with respect to $\cA$ and
\[
\left(\reallywidehat{\gamma_{a\mu+b\nu+c \omega }}\right)_{pp}= |a|^2 \reallywidehat{\gamma_{\mu}} \mbox{ and }
\left(\reallywidehat{\gamma_{a\mu+b\nu+c\omega}}\right)_{ac}=  |b|^2 \reallywidehat{\gamma_{\nu}} \mbox{ and }
\left(\reallywidehat{\gamma_{a\mu+b\nu+c\omega}}\right)_{sc}=
|c|^2 \reallywidehat{\gamma_{\omega}}\,.
\]
\qed
\end{corollary}

This allows us to give the  following characterisation for existence
of an Eberlein decomposition.

\begin{theorem} Let $\omega$ be a translation bounded measure whose autocorrelation $\gamma$ exists with respect to $\cA$. Then, the following are equivalent:
\begin{itemize}
  \item[(i)] The refined Eberlein decomposition of $\omega$ exists with respect to $\cA$.
  \item[(ii)]
There exists some measures $ \omega_{1},  \omega_{2}, \omega_3$ such
that $\omega=\omega_1+\omega_2+\omega_3$, the autocorrelations of
$\gamma_1, \gamma_2, \gamma_3$ of $\omega_1, \omega_2, \omega_3$
exists with respect to $\cA$ and
$\widehat{\gamma_1},\widehat{\gamma_2}, \widehat{\gamma_3}$ are pure
point, absolutely continuous and singular continuous, respectively.
\end{itemize}
\end{theorem}
\begin{proof}
(i) $\Longrightarrow$ (ii) is obvious.

(ii) $\Longrightarrow$ (i) is Corollary~\ref{cor3}.
\end{proof}

Let us note here in passing that this result simply says that to get a refined Eberlein decomposition we only need to write our measure $\omega$ as the sum of three measures
of pairwise distinct spectral purity. If this is the case, then the diffractions of three measures give us exactly the Lebesgue decomposition of the diffraction of the initial measure.

\section{Application: Orthogonality of dynamical systems}\label{section-dynamical-systems}
In this section we assume that  $G$ is not only  a locally compact
Abelian group but also $\sigma$-compact. Let $\lambda$ be the Haar
measure on $G$. A dynamical system over $G$ is a triple
$(X,\alpha,m)$ consisting of a compact space $X$, a
continuous  action $\alpha$ of $G$ on $X$ and an $\alpha$
invariant probability measure $m$. The dynamical system is called
\textit{ergodic} if any $\alpha$-invariant measurable set has measure in
$\{0,1\}$. Whenever $(X,\alpha,m)$ is an ergodic dynamical
system and $\cA$ is a van Hove sequence, we say that
\textit{Birkhoff theorem holds along $\cA$} if  for any integrable
$f$ the equality
\[
\lim_{n\to \infty} \frac{1}{|A_n|} \int_{A_n} f(\alpha_t x) \dd t = \int_X f(x) \dd m
\]
holds for $m$ almost every $x\in X$. Any ergodic dynamical
system admits a van Hove sequence along which Birkhoff ergodic
theorem holds \cite{Lin}.

Any dynamical system comes with a unitary representation $T=
T^X$ of $G$ on $L^2 (X,m)$ given by $T_t f = f\circ
\alpha_t$. The spectral measure of $f\in L^2 (X,m)$ is denoted
by $\varrho_f$.

For $f :X \longrightarrow \CC$ and $x\in X$ we define
 $f_x : G\longrightarrow \CC$ by
\[
f_x (t) :=f(\alpha_t x)\,.
\]

We now provide a variant of \cite{Len2}:

\begin{proposition} Assume that $G$ is $\sigma$-compact and has a dense countable set.
Consider an ergodic dynamical system $(X,\alpha,m)$. Let $\cA$
be a van Hove sequence along which the Birkhoff ergodic theorem
holds. Then, for any  $f \in C(X)$
and almost every $x\in X$
the reflected Eberlein convolution $\lb f_x, f_x \rb_{\cA}$ exists and
\[
\lb f_x, f_x \rb_{\cA}(t)=\widecheck{\sigma_f}(t) \,.
\]\qed
\end{proposition}
\begin{proof}
Let $D= \{ t_n : n \in \NN \}$ be a countable dense set in $G$. Then, there exists some set $X_n \subseteq X$ of full measure such that, for all $x \in X_n$ we have
\[
\lim_i \frac{1}{|A_i|} \int_{A_i} f_x (s) \overline{f_x (s-t_n)} \dd s = \langle f , T_{t_n} f \rangle = \widecheck{\sigma_f}(t_n) \,.
\]
Then $Y =\bigcap_n X_n$ has full measure in $X$ and for all $x\in Y$ and all $n$ we have
\[
\lim_i \frac{1}{|A_i|} \int_{A_i} f_x (s) \overline{f_x (s-t_n)} \dd s = \langle f , T_{t_n} f \rangle = \widecheck{\sigma_f}(t_n) \,.
\]
Then, $\lb f,f \rb_{\cA}$ exists by Lemma~\ref{lem4} and for all $n$ we have
\[
\lb f_x, f_x \rb_{\cA}(t_n) = \widecheck{\sigma_f}(t_n) \,.
\]
The claim follows now by continuity and denseness of $D$.
\end{proof}

Given this proposition our main result on orthogonality has the
following immediate consequence.

\begin{corollary}\label{lem1}Assume that $G$ is $\sigma$-compact and has a dense countable set.  Let $(X,\alpha,m)$ and $(X',\alpha',m')$ be ergodic
dynamical systems.  Let $\cA$ be a van Hove sequence along which the
Birkhoff ergodic theorem holds (for both systems). Let $f\in L^2 (X,m)$ and
$g\in L^2 (X',m')$ be given with $\sigma_{f}\perp
\sigma_{g}$. Then, for $m \times m'$ almost every
$(x,x')\in X\times X'$ we have
\[
\lb f_{x}, g_{x'}\rb_{\cA} = 0 \,.
\]\qed
\end{corollary}

As an immediate consequence, we get the following.

\begin{corollary} In the situation of Lemma~\ref{lem1}, assume furthermore
that  $(X,\alpha,m)$ has pure
point spectrum (i.e. all spectral measures are pure point measures)
and $(X,\alpha,m')$ is weak mixing (i.e. the spectral
measures  of all $f\perp 1$ are continuous). Let $f\in C(X)$ and $g\in C(X')$ be given. Then, for
$m \times m'$ almost every $(x,x')$ the equality
\begin{align*}
\lb f_{x}, g_{x'}\rb_{\cA}  =\langle 1,g\rangle \langle f,1\rangle 1_{G}
  &= \left( \int_{X'} \overline{g} \dd m' \right) \left( \int_{X} f \dd m \right) 1_{G}
\end{align*}
holds.
\end{corollary}

\begin{proof} We can decompose $g = \langle g,1  \rangle 1 + h$
with $h \perp 1$.  By $h\perp 1$ and, as $(X',\alpha',m')$
is weakly mixing, the spectral measure $\sigma_h$ is continuous.
Hence, $\gamma_{h_{x'}} = \sigma_h$ is continuous for
almost every $x'\in X'$. Now, the result follows directly from the preceding corollary.
\end{proof}

\section{Existence of the reflected Eberlein convolution for Besicovitch almost periodic measures}
In this section we discuss existence of the reflected Eberlein convolution for Besicovitch almost periodic measures. In particular, we determine the orthogonal complement of such measures.

Throughout this section we consider a $\sigma$-compact locally compact Abelian group $G$ and  we let a van Hove sequence $\cA$  be given.

A translation bounded measure $\mu$ is called \textit{Besicovitch
almost periodic} if $\mu\ast\varphi$ belongs to the
space of Besicovitch 2- almost periodic functions discussed above (in Example
\ref{ex-bes}) for all $\varphi \in\Cc (G)$.

Besicovitch almost periodic measures admit many reflected Eberlein
convolutions. In fact, we have the following characterization.

\begin{theorem} Let $G$ be a $\sigma$-compact LCAG. Then, the following statements are equivalent for the translation bounded $\nu$.

\begin{itemize}

\item[(i)] The reflected Eberlein convolution $\lb \mu,\nu \rb_{\cA}$ exists for all Besicovitch almost periodic measures $\mu \in \cM^\infty(G)$.
\item[(ii)] The Fourier--Bohr coefficients $a_\xi (\nu)$ exist for all $\xi \in \widehat{G}$.
\end{itemize}
Moreover, for $\mu,\nu$ satisfying conditions (i) and
(ii) the measure $\lb \mu,\nu\rb_{\cA}$ is the unique strongly
almost periodic measure satisfying the generalized consistency phase
property:
\[
a_{\xi}( \lb \mu,\nu\rb_{\cA})=a_\xi (\mu) \overline{a_\xi (\nu)} \qquad \mbox{ for all } \xi \in \widehat{G} \,.
\]

\end{theorem}
\begin{proof}
(i) $\Longrightarrow$ (ii) follows immediately from the fact that $\xi \theta_{G}$ is Besicovitch almost periodic.

(ii) $\Longrightarrow$ (i) and the last claim follows from \cite[Proposition~4.38]{LSS3}  (compare \cite[Proposition~3.26]{LSS3}).
\end{proof}

From the previous result and the fact that a strongly almost
periodic measure is zero if and only if all its Fourier--Bohr
coefficients are 0 (see for example \cite[Corollary~4.6.10]{MoSt})
we obtain the following result on orthogonality.

\begin{corollary}[Orthogonal complement of Besicovitch almost periodic measures]\label{cor-BAP} Let $\nu$ be a translation bounded measure. Then $\lb \mu,\nu\rb_{\cA} =0$ for all Besicovitch almost periodic $\mu$ if and only if all Fourier--Bohr coefficients $a_{\xi}^\cA(\nu)$ of $\nu$ exist and vanish.
\end{corollary}

\begin{appendix}
\section{Some background on translation bounded measures}

Let $V\subset G$ be an open subset of $G$ with compact closure. Let
$C>0$. We denote by $\cM_{C,V} (G)$ the set of those measures  $\mu$
with
\[
\|\mu \|_{V}:=\sup_{t\in G} |\mu|(t+V) \leq C \,.
\]
Let now  $W$  be another  open relative compact set. Then, we can
cover $\overline{W}$ by finitely many translates of $V$. Hence,
there exists a $D>0$ with
\[
\cM_{C,V}(G)\subset \cM_{D,W}(G)\,.
\]
In particular,
\[
\cM^\infty (G):=\bigcup_{C>0} \cM_{C,V} (G)
\]
is independent of $V$ open and relatively compact. The elements of
$\cM^\infty (G)$ are called translation bounded measures.

We recall the following statement of \cite{BL}.

\begin{proposition}[Compactness of $\cM_{C,V}(G)$] For any $C>0$ and $V\subset G$ open with compact
closure the set $\cM_{C,V} (G)$ is compact.

Moreover, if $G$ is second countable, the vague topology is metrisable on $\cM_{C,V} (G)$.\qed
\end{proposition}

We next gather  the following (well-known) characterizations of
vague convergence of measures.

\begin{lemma}[Characterization of vague convergence]
For a  net $(\mu_i)_{i \in I} \in\cM_{C,V}(G)$ the following statements are
equivalent:
\begin{itemize}
\item[(i)] $\mu_i$ converges vaguely to some $\mu \in \cM_{C,V}(G)$.
\item[(ii)]  $\mu_i(\varphi)$ converges for any $\varphi \in\Cc (G)$.
\item[(iii)]   $\mu_i\ast \varphi (0)$ converges for any $\varphi \in\Cc(G)$.
\item[(iv)]  $\mu_i\ast\varphi \ast \widetilde{\psi}(0)$ converges for any $\varphi,\psi\in\Cc (G)$.
\end{itemize}
\end{lemma}
\begin{proof}
(i)$\Longrightarrow$(iv): This is clear as
\[
\mu_i \ast \varphi \ast \widetilde{\psi}(0) = \mu_i (\varphi^\dagger \ast
\overline{\psi})
\]
 and $\varphi^\dagger \ast \overline{\psi}\in\Cc (G)$.

(iv)$\Longrightarrow$ (iii):

Let $\varepsilon>0$ be given. Let $W$
be an arbitrary open relatively compact set containing the support
of $\varphi$. Chose $D>0$ with $\cM_{C,V}\subset \cM_{D,W}$. Chose
$\psi \in\Cc (G)$ such that $\varphi \ast \widetilde{\psi}$ is
supported in $W$ and
\[
\|\varphi - \varphi \ast
\widetilde{\psi}\|_\infty \leq \varepsilon \,.
\]
Then
\[
|\mu_i \ast \varphi \ast \widetilde{\psi} - \mu_i \ast
\varphi\|_\infty \leq 2 \varepsilon D \,.
\]
This easily gives the desired implication.

(iii)$\Longrightarrow$(ii): For any measure $\mu$ and $\varphi
\in\Cc (G)$ we have $\mu (\varphi^\dagger) = \mu\ast \varphi
(0)$. This easily gives the desired implication.

(ii)$\Longrightarrow$(i): We can define $\mu : \Cc(G)\longrightarrow \CC, \mu (\varphi) =\lim_i \mu_i (\varphi)$.
Then, $\mu$ is linear as it is a pointwise limit of linear functionals.

Next, fix a compact set $K \subseteq G$ and let $U$ be a pre-compact open set such that $K \subseteq U$. Let $D>0$ be so that
\[
\cM_{V,C} \subseteq \cM_{U,D} \,.
\]
Now, for all $\varphi \in \Cc(G)$ with $\sup(\varphi) \subseteq K$ we have
\begin{align*}
  |\mu(\varphi)| &= \left| \lim_i \mu_i(\varphi)\right| \leq \sup_{i}\left| \mu_i(\varphi)\right| \leq \sup_{i}\left| \mu_i \right| (|\varphi|) \\
  &\leq \sup_{i}\left| \mu_i \right| (\| \varphi\|_\infty 1_{K})\leq \| \varphi\|_\infty \sup_{i}\left| \mu_i \right| (U) \leq D \| \varphi\|_\infty \,.
\end{align*}
This shows that $\mu$ is a measure. By its definition, $\mu_i$ converges vaguely to $\mu$.

Finally, since $U$ is open, for all $t \in G$ we have by the inner regularity of $|\mu|$:
\begin{align*}
  |\mu|(t+V)| &= \sup\{ |\mu|(\varphi): \varphi \in \Cc(G) , \varphi \leq 1_{t+V} \} = \sup\{ |\mu(\psi)|: \psi \in \Cc(G) , |\psi| \leq 1_{t+V} \} \\
  &= \sup\{ |\lim_i \mu_i(\psi)|: \psi \in \Cc(G) , |\psi| \leq 1_{t+V} \} \\
  &\leq  \sup\{ |\mu_i(\psi)| : \psi \in \Cc(G) , |\psi| \leq 1_{t+V}, i \in I \} \\
  &\leq  \sup\{ |\mu_i|(|\psi|) : \psi \in \Cc(G) , |\psi| \leq 1_{t+V}, i \in I \} \\
   &\leq  \sup\{ |\mu_i|(1_{t+V}) : i \in I \} \leq \sup \{ \| \mu_i \|_{V} : i \in I \}\leq C  \,.
\end{align*}
Taking the supremum over all $t \in G$ we get $\mu \in \cM_{C,V}(G)$.
\end{proof}

We now turn to a universal bound (compare \cite[Section~9]{LR} or \cite{Martin}).

\begin{proposition}[A universal bound] Let $V\subset  G$ be a nonempty, open relatively compact set.
 Then, for all $\nu \in \cM^\infty(G)$ and every relatively compact $B\subset G$ we have
\[
|\nu| (B) \leq \frac{|B - V|}{|V|} \, \| \nu \|_{V} \,.
\]
\end{proposition}
\begin{proof} A direct calculation shows
$1_{B} \leq \frac{1}{|V|} 1_{B - V} \ast 1_V.$ This
gives
\begin{align*}
|\nu| (B) &\leq     \frac{1}{|V|} \int_G \int_G 1_V (t -s) \,
1_{B - V} (s) \,\dd s\,\dd |\nu|(t)\\
&\stackrel{\mbox{Fubini}}{=\joinrel=\joinrel=\joinrel=\joinrel=\joinrel=\joinrel=}  \frac{1}{|V|}\int_G 1_{B-V} (s) \left(\int_G
1_V (t -s)\dd |\nu| (t)\right) \dd s \\
&\leq {\| \nu \|_{V}} \frac{1}{|V|}\int_G 1_{B-V} (s) \dd s = \frac{|B - V|}{|V|} \| \nu \|_{V} \,.
\end{align*}
This finishes the proof.
\end{proof}

As a consequence, we get:
\begin{proposition}\label{prop3} Let $\nu \in \cM^\infty (G)$ be given.
Let $(A_i)$ be a van Hove net. Then,
\[
\limsup_i \frac{|\nu|(A_i)}{|A_i|} \leq \frac{ {\| \nu \|_{V}} }{|V|}< \infty \,.
\]
\end{proposition}
\begin{proof} From the previous proposition we infer
\[
|\nu| (A_i) \leq \frac{\| \nu \|_{V} }{|V|} |A_i - V| \,.
\]
Now, we have
\[
(A_i -V) \subset  \overline{(A_i - \overline{V})}\setminus A_i
\cup A_i \subset \partial^{\overline{V}}  A_i \cup A_i \,.
\]
This immediately gives the desired statement.
\end{proof}

\begin{corollary}\label{cor5} Let $\mu,\nu$ be translation bounded measures, $(A_i)$ a van Hove
net and $V=-V \subseteq G$ be a fixed open, relatively compact set. Then, there exists an index $i_0$ and some $\kappa$ such that, for all
$\mu, \nu \in \cM^\infty(G)$ and all $i > i_0$ we have
\[
\| \frac{1}{|A_i|}\mu|_{A_i} \ast \widetilde{\nu|_{A_i}} \|_{V} \leq \kappa \| \mu \|_{V} \| \nu \|_V
\]
In particular,
\[
\{ \frac{1}{|A_i|}\mu|_{A_i} \ast \widetilde{\nu|_{A_i}} : \mu, \nu \in \cM_{C,V}, i > i_0 \} \subseteq \cM_{\kappa C,V} \,.
\]
\end{corollary}
\begin{proof} Note first that since $V=-V$ we have
\[
\| \widetilde{\nu|_{A_i}} \|_{V} \leq \| \nu \|_{V} \,.
\]
Moreover, by Proposition~\ref{prop3}, there exists some $i_0$ such that, for all $i >i_0$ we have
\[
\left|\frac{1}{|A_i|} \mu_{A_i}\right|(G)= \frac{1}{|A_i|}|\mu|(A_i) \leq \frac{2C}{|V|} \,.
\]
The claim follows now immediately from \cite[Lemma~6.1]{NS20}.
\end{proof}

\section{Universal van Hove nets}
In this section we prove the existence of "universal" van Hove nets, along which all reflected Eberlein convolutions of functions and measures exist, as well as all Fourier--Bohr coefficients.

To make the proofs easier to follow we do them in 3 steps.

\begin{lemma}[Universal van Hove net for Fourier--Bohr coefficients] Let $G$ be a LCAG and $\{ A_i \}_{i \in I}$ a van Hove net. Then, there exists a subnet $ \{ B_j \}_{j \in J}$ of $A_i$ such that, for all $f \in L^\infty(G)$ and all $\chi \in \hat{G}$ the Fourier--Bohr coefficient $a_{\chi}^{\cB}(f)$ exists.
\end{lemma}
\begin{proof}  By linearity of the Fourier--Bohr coefficient, it suffices to consider $f\in L^\infty (G)$ with $\|f\|_\infty \leq 1$. Let us consider
\[
X:= \{ (f, \chi) : f \in L^\infty(G), \| f \|_\infty \leq 1, \chi \in \widehat{G} \} \,.
\]
Then,
\[
\{ (\frac{1}{|A_i|} \int_{A_i} \chi(t) f(t) \dd t)_{(f, \chi) \in X} \} _{i \in I}
\]
is a net in $D^{X}$ where $D= \{ z \in \CC : |z| \leq 1 \}$. This is compact by Tychonoff's theorem. Therefore, this net has a convergent subnet $(y_j)_{j \in J}$.
This means that there exists a monotone final function $h : J \to I$ such that for all $j \in J$ we have
\[
y_j= \bigl(\frac{1}{|A_{h(j)}|} \int_{A_{h(j)}} \chi(t) f(t) \dd t\bigr)_{(f, \chi) \in X}
\]
Defining $B_j=A_{h(j)}$ for all $j \in J$ gives the claim for all $f \in L^\infty(G)$ with $\| f \|_\infty \leq 1$.
\end{proof}

\begin{lemma}[Universal van Hove net for reflected Eberlein convolution of functions] Let $G$ be a LCAG and $\{ A_i \}_{i \in I}$ a van Hove net. Then, there exists a subnet $ \{ B_j \}_{j \in J}$ of $A_i$ such that, for all $f, g \in L^\infty(G)$ reflected Eberlein convolution $\lb f, g \rb_{\cB}$ exists.
\end{lemma}
\begin{proof} By linearity of the reflected Eberlein convolution in both arguments, it suffices to consider $f,g\in L^\infty (G)$ with $\|f\|_\infty, \|g\|_\infty \leq 1$.
Let
\[
X:= \{ (f, g, t) : f, g \in L^\infty(G), t \in G, \| f \|_\infty \leq 1,  \|g \|_\infty \leq 1 \} \,.
\]
Then,
\[
\{ (\frac{1}{|A_i|} \int_{A_i} f(s) \overline{g(s-t)} \dd s)_{(f, g, t) \in X} \} _{i \in I}
\]
is a net in $D^{X}$ where $D= \{ z \in \CC : |z| \leq 1 \}$. This is compact by Tychonoff's theorem. Therefore, this net has a convergent subnet $(y_j)_{j \in J}$.

Similarly to the above, this yields a subnet $\cB$ of $\cA$ such that for all $f,g \in L^\infty(G)$ with $\| f \|_{\infty} \leq 1, \| g \|_\infty \leq 1$ the reflected Eberlein convolution $\lb f,g \rb_{\cB}$ exists.
\end{proof}

\begin{lemma}[Universal van Hove net for reflected Eberlein convolution of measures] Let $G$ be a LCAG and $\{ A_i \}_{i \in I}$ a van Hove net. Then, there exists a subnet $ \{ B_j \}_{j \in J}$ of $A_i$ such that, for all $\mu, \nu \in \cM^\infty(G)$ the reflected Eberlein convolution $\lb \mu, \nu \rb_{\cB}$ exists.
\end{lemma}
\begin{proof} This could be shown similarly to the arguments given in the proofs of  the proceeding two lemmas. However, it is also a direct consequence of the preceding lemma and our discussion of the Eberlein convolution in Lemma \ref{lemma-relationsship}.
\end{proof}

Applying the three results in succession, we get:

\begin{theorem}[Existence of universal van Hove net for Fourier--Bohr coefficients]\label{thm:univ}
Every LCAG admits a van Hove net with the following properties:
\begin{itemize}
  \item[(a)] For all $f \in L^\infty(G)$ the Fourier--Bohr coefficient $a_{\chi}^\cA(f)$ exists.
  \item[(b)] For all $f,g \in L^\infty(G)$,  $\lb f,g \rb_{\cA}$ exists.
  \item[(c)] For all $\mu,\nu \in \cM^\infty(G)$,  $\lb \mu,\nu \rb_{\cA}$ exists.
  \item[(d)] For all $\mu \in \cM^\infty(G)$,  the autocorrelation $\gamma = \lb \mu ,\mu \rb_{\cA}$ exists with respect to $\cA$.
\end{itemize}
Furthermore, any van Hove net has a subnet with these properties.

\qed
\end{theorem}

\begin{definition}We will refer to any net satisfying Theorem~\ref{thm:univ} as an \emph{universal van Hove net}.
\end{definition}

For such nets, \cite{LSS3} gives:

\begin{corollary} Let $\cA$ be an universal van Hove net in $G$. Then,
\begin{itemize}
  \item[(a)] $\lb \, ,  \, \rb_{\cA}$ is a mapping from $L^\infty(G) \times L^\infty(G)$ into $WAP(G)$.
  \item[(b)] $\lb \, ,  \, \rb_{\cA}$ is a mapping from $\cM^\infty(G) \times \cM^\infty(G)$ into $\WAP(G)$.
\end{itemize}\qed
\end{corollary}

Let us conclude by pointing out that, when working with an universal van Hove net, all the assumptions on
the existence of the reflected Eberlein convolution throughout Section~\ref{section-construction} can be dropped.
While one could write the results that way, in general one works with an explicit van Hove sequence or net, and we expect that
in general an universal van Hove net cannot be constructed explicitly.

\end{appendix}

 \subsection*{Acknowledgments} This work was started when NS visited Jena University, and he would like to thank the Mathematics Department for hospitality.
NS was supported by the Natural Sciences and Engineering Council of Canada via grant 2020-00038, and he is grateful for the support.

\end{document}